\documentclass[12pt, reqno]{amsart}
\pdfoutput=1
\usepackage{versions}

\includeversion{withcode}
\excludeversion{withoutcode}
\usepackage{fullpage}

\usepackage{amssymb}
\usepackage{verbatim}
\usepackage{calrsfs}
\usepackage{graphicx}
\usepackage{epsfig}
\usepackage{color}
\usepackage[all]{xy}

\numberwithin{equation}{section} 

\newtheorem{thm}{Theorem}[section]
\newtheorem*{thm*}{Theorem}
\newtheorem{lem}[thm]{Lemma}
\newtheorem{cor}[thm]{Corollary}
\newtheorem{prop}[thm]{Proposition}

\newtheorem*{Bogomolov}{Geometric Bogomolov Conjecture}
\newtheorem*{Gap}{Height Gap Principle}

\newtheorem*{conjecture*}{Conjecture}
\newtheorem{conjecture}[thm]{Conjecture}
\theoremstyle{remark} \newtheorem*{question*}{Question}
\theoremstyle{definition} 
\theoremstyle{remark} 
\theoremstyle{remark} 
\theoremstyle{remark} \newtheorem{remark}[thm]{Remark}
\theoremstyle{remark} 
\theoremstyle{definition} 
\theoremstyle{remark}  
\theoremstyle{definition}

\newcommand{\OO}{\mathcal{O}}    
\newcommand{\ZZ}{\mathbb{Z}}     
\newcommand{\RR}{\mathbb{R}}     
\newcommand{\PP}{\mathbb{P}}      
\newcommand{\QQ}{\mathbb{Q}}      
\newcommand{\CC}{\mathbb{C}}      

\newcommand{\be}{\begin{equation}}
\newcommand{\ee}{\end{equation}}
\newcommand{\benn}{\begin{equation*}}
\newcommand{\eenn}{\end{equation*}}
\newcommand{\ba}{\begin{aligned}}
\newcommand{\ea}{\end{aligned}}
\newcommand{\bbm}{\begin{bmatrix}}
\newcommand{\ebm}{\end{bmatrix}}
\newcommand{\bpm}{\begin{pmatrix}}
\newcommand{\epm}{\end{pmatrix}}
\newcommand{\bi}{\begin{itemize}}
\newcommand{\ei}{\end{itemize}}
 
\newcommand{\Spec}{\operatorname{Spec}}
\newcommand{\ord}{\operatorname{ord}}
\newcommand{\Div}{\operatorname{Div}}    
\newcommand{\Pic}{\operatorname{Pic}}
\newcommand{\Jac}{\operatorname{Jac}}       

\newcommand{\simarrow}{\stackrel{\sim}{\rightarrow}}    
\newcommand{\alg}[1]{\overline{#1}}   
\newcommand{\dint}{\int \!\!\! \int}   

\newcommand{\Zepsilon}[1]{\epsilon\left( #1\right)}
\newcommand{\Ztau}[1]{a \left(#1\right)}

\newcommand{\Zphi}[1]{\varphi\left(#1\right)}

\newcommand{\Zg}[1]{g \left(#1\right)}
\newcommand{\agam}{\alg{\Gamma}}

\title[Geometric Bogomolov Conjecture]{The Geometric Bogomolov Conjecture \\ for Small Genus Curves 
}
\author{X.W.C. Faber}
\address{
Department of Mathematics and Statistics \\
McGill University \\
Montr\'eal, QC  H3A 2K6 \\ 
CANADA} 
\email{xander@math.mcgill.ca}
\urladdr{http://www.math.mcgill.ca/xander/}


\begin{document}
	\begin{abstract}
		The Bogomolov Conjecture is a finiteness statement about algebraic points of small height on a 
		smooth complete curve defined over a global field. We verify an effective form of the Bogomolov 
		Conjecture for all curves of genus at most $4$ over a function field of characteristic zero. 
		We recover the known result for genus $2$ curves and in many cases improve upon the known bound for 
		genus $3$ curves. For many curves of genus~$4$ with bad reduction, the conjecture was 
		previously unproved. 
	\end{abstract}

\maketitle



\section{Introduction}

\subsection{The Conjecture and Main Theorem}
	
	Fix an algebraically closed field $k$ of characteristic zero and a smooth proper connected curve $Y/k$. Define $K$ to be the field of rational functions on $Y$. Let $C$ be a smooth proper geometrically connected curve of genus at least $2$ over the function field $K$. Choose a divisor $D$ of degree~1 on $\alg{C} = C \times_K \alg{K}$ and consider the embedding of $C$ into its Jacobian $\Jac(C) = \Pic^0(C)$ given on geometric points by $j_D(x) = [x] - D$. Define
	\benn
		a'(D) = \liminf_{x \in C(\alg{K})} \hat{h} \left(j_D(x)\right), 
	\eenn
where $\hat{h}$ is the canonical N\'eron-Tate height on the Jacobian associated to the symmetric ample divisor $\Theta + [-1]^*\Theta$. As $C(\alg{K})$ may not be countable, the liminf is taken to mean the limit over the directed set of all cofinite subsets of $C(\alg{K})$ of the infimum of the heights of points in such a subset. Recall that $C$ is called \textbf{constant} if there is a curve $C_0$ defined over the constant field $k$ and a finite extension $K' / K$ such that $C_{K'} = C_0 \times_k K'$. Write $\Div^1\left(\alg{C}\right)$ for the set of divisors of degree~1 on $\alg{C}$. We wish to investigate the 
\begin{Bogomolov}[\cite{Bogomolov_Conjecture}]
 	If $C$ is not a constant curve, then 
		\[
			\inf_{D \in \Div^1\left(\alg{C}\right)}a'(D) > 0.
		\]
\end{Bogomolov}

	In response, we have proved the following result:

\begin{thm} \label{Thm: Main Theorem Imprecise}
	Let $K$ be a function field of characteristic zero. Then the Geometric Bogomolov Conjecture is true for all 
	curves $C / K$ of genus $2 \leq g \leq 4$. Moreover, if $C$ is not a constant curve, 
	there is an effectively computable positive lower bound for $a'(D)$ that is uniform in $D$.
\end{thm}

A more precise form of this result is given by Theorem~\ref{Thm: Main Theorem} below. 

	It is worth noting that the dependence on $D$ in the Geometric Bogomolov Conjecture is superficial. Indeed, it follows from \cite[Thm.~5.6]{Zhang_Admissible_Pairing_1993} that for any degree-1 divisor $D$,
	\benn
		a'(D) \geq \frac{1}{2}a'\left(\frac{K_C}{2g-2}\right),
	\eenn
where $K_C$ is a canonical divisor on $C$ and $K_C / (2g-2)$ is a degree-1 divisor with rational coefficients. Thus is suffices to obtain a positive lower bound for the single divisor $\xi = K_C / (2g-2)$. Now we can reformulate the Geometric Bogomolov Conjecture in the following more intuitive form:

\begin{Gap}
	If $C$ is not a constant curve and $\xi =  K_C / (2g-2)$, then there exists $\varepsilon > 0$ such that for any $x \in C(\alg{K})$, 
		\[
			\hat{h}\left(j_{\xi}(x)\right) \not= 0 \Longrightarrow \hat{h}\left(j_{\xi}(x)\right) \geq \varepsilon.
		\]
\end{Gap}

	The Height Gap Principle is true for all curves of genus at most $4$ by Theorem~\ref{Thm: Main Theorem Imprecise}. However, the theorem does not provide any information on the algebraic point of smallest positive height, so we cannot effectively choose $\varepsilon$ in the statement of the Height Gap Principle. 
	
\subsection{A More Precise Statement and Some History}
\label{Sec: Statement}

	We now provide a more precise statement of Theorem~\ref{Thm: Main Theorem Imprecise}. We continue to use the notation from the previous section, but let us stress that we still assume the curve $C$ has genus at least $2$. By the Semistable Reduction Theorem we may pass to a finite extension field $K'$ over which $C_{K'} = C \times_K K'$ has semistable reduction. One can do this effectively by choosing $K'$ so that all of the $12$-torsion points of $\Jac(C)$ are rational. For references on semistable reduction theory, including proofs of these facts, see \cite[expos\'e $n^{\circ}$ 1]{Asterisque_86} and \cite[\S9.3.3, 10.3, 10.4]{Qing_Liu_Algebraic_Geometry}.
	
	Let $Y'/k$ be a smooth proper curve with field of rational functions $K'$. To say that $C_{K'}$ has semistable reduction means there is a projective surface $X' / k$ and a proper flat morphism $f: X' \to Y'$ so that
	\begin{itemize}
		\item $f$ has generic fiber isomorphic to $C_{K'}$.
		\item The fibers of $f$ are connected and reduced with only nodal singularities. 
		\item If $Z$ is an irreducible component of a fiber and $Z \cong \PP^1$, then $Z$ meets the other 
		components of the fiber in at least $2$ points.
	\end{itemize}
If we assume further that $X'$ is a smooth surface (over $k$), then such a morphism $f: X' \to Y'$ is unique up to canonical isomorphism, and it may be characterized as the \textbf{minimal regular model} of $C_{K'}$ over $Y'$. 

	We divide the fiber singularities of $f$ into different types as follows. Choose a point $y \in Y'(k)$. The \textbf{partial normalization} of the fiber $f^{-1}(y)$ at a node $p$ is the $k$-scheme given by resolving the singularity at $p$. We say $p$ is of \textbf{type~0} if the partial normalization at $p$ is connected. Otherwise, the partial normalization at $p$ has two connected components. If the fiber $f^{-1}(y)$ has arithmetic genus~$g$, then one component of the partial normalization has arithmetic genus~$i$ and the other has arithmetic genus~$g-i$. We may assume $i \leq g-i$. In this case we say the node $p$ is of \textbf{type~$i$}. Let $\delta_i = \delta_i(X' / Y')$ be the total number of nodes of type $i$ in all fibers. By uniqueness of the minimal regular model, the numbers $\delta_i$ are well-defined invariants of $C_{K'}$.

\newpage

\begin{thm} \label{Thm: Main Theorem}
	Let $K$ be a function field of characteristic zero over the algebraically closed constant field $k$.
	Let $C/K$ be a smooth proper geometrically connected curve of genus $g$ with $2 \leq g \leq 4$. 
	Suppose that $K'$ is a finite extension of $K$ over which $C_{K'}$ has semistable reduction. Define
	$Y'$ to be the smooth curve with function field $K'$, $f: X' \to Y'$ the minimal regular (semistable) model
	of $C_{K'}$, and $d = [K':K]$. Then
		\benn
			\inf_{D \in \Div^1\left(\alg{C}\right)}a'(D) \geq  
				\begin{cases} {\displaystyle \ \ \ \frac{3}{d(g-1)}} & \text{if $f$ is smooth} \\
						{\displaystyle \frac{1}{2d(2g+1)}\left(c(g) \delta_0
					+\sum_{i \in (0, g/2]} \frac{2i(g-i)}{g}\delta_i \right)} & \text{unconditionally}
				\end{cases}
		\eenn
	with $c(2)=\frac{1}{27}$, $c(3)= \frac{2}{81}$, and $c(4) = \frac{1}{36}$. In particular, the Geometric
	Bogomolov Conjecture is true for curves of genus at most $4$.
\end{thm}

	Results of this type have been known for about a dozen years now, and the novelty of the present paper is twofold. First, it gives an algorithm for verifying the Geometric Bogomolov Conjecture for all curves of a fixed genus. Second, the method given here seems to admit a generalization, although we do not yet completely understand it. We are in a position to make the following strong effectivity conjecture, which we do not claim is in any way optimal:
	
\begin{conjecture}
	Let $C/K$ be a smooth proper geometrically connected curve of genus $g \geq 2$ that
	admits semistable reduction over $K$. Then
		\benn
			\inf_{D \in \Div^1\left(\alg{C}\right)}a'(D) \geq  \frac{1}{2(2g+1)}\left( \frac{g-1}{27g}\delta_0
					+\sum_{i \in (0, g/2]} \frac{2i(g-i)}{g}\delta_i \right).
		\eenn 
\end{conjecture}

	Now let us summarize the previous results relevant to the conjecture. For simplicity, let us assume that $C$ has semistable reduction over $K$ (so that $d = 1$ in the theorem) and that $f:X \to Y$ is the minimal regular model. 
	
		\begin{enumerate}
		
			\item (Par\v{s}in, \cite{Parsin_Algebraic_Curves_FF_1968}) For any $g \geq 2$, if $f$ is smooth
				then the relative dualizing sheaf $\omega_{X/Y}$ is ample and its self-intersection number
				satisfies $\omega_{X/Y} . \omega_{X/Y} \geq 12$. By 
				\cite[Thm.~5.6]{Zhang_Admissible_Pairing_1993} this implies
					\[
						\inf_{D \in \Div^1\left(\alg{C}\right)}a'(D) \geq \frac{3}{g-1}.
					\]
				(See also Theorem~\ref{Thm: Admiss Pairing} below and the subsequent discussion.)
			\item (Moriwaki, \cite{Moriwaki_Bogomolov_Genus_2_1996}) If the genus of $C$ is $2$, then
				$f$ is not smooth and
				\[
					\inf_{D \in \Div^1\left(\alg{C}\right)}a'(D) \geq \frac{1}{270}\left(\delta_0 
						+ \delta_1\right).
				\]
				
			\item (Moriwaki, \cite{Moriwaki_Bogomolov_Irreducible_Fibers_1997,
				Moriwaki_Bogomolov_Stable_Trees_1998}) If the dual graph of each
				closed fiber of the stable model of $f$ consists of a tree with loop edges attached, then
				\[
					\inf_{D \in \Div^1\left(\alg{C}\right)}a'(D) \geq \frac{1}{2(2g+1)} 
							\left( \frac{g-1}{6g}\delta_0 + \sum_{i=1}^{[g/2]}
								\frac{2i(g-i)}{g}\delta_i \right).
				\]
			
			\item (Yamaki, \cite{Yamaki_Bogomolov_Hyperelliptic_2008}) If $C$ is hyperelliptic of genus
				$g \geq 3$ and the hyperelliptic involution $\iota$ extends to the family $f$, then
		
				{\small
				\benn
					\ba
						\inf_{D \in \Div^1\left(\alg{C}\right)}a'(D) &\geq \frac{1}{4g(2g+1)} 
							\bigg[\frac{2g-5}{12}\xi_0 
						 + \sum_{j=1}^{[(g-1)/2]} \Big\{ (g-1-j) - 1 \Big\}\xi_j \\ &
						 \qquad - \alpha(g) \Big\{\delta_0 - \xi_0 \Big\}
						 + \sum_{i=1}^{[g/2]}4i(g-i)\delta_i\Bigg],			
					\ea	
				\eenn
				}
				where $\xi_0 = \xi_0(X/Y)$ is the number of nodes of type~0 that are fixed by the involution and
				$\xi_j = \xi_j(X/Y)$ is the number of \textit{pairs} of nodes $\{p, \iota(p)\}$ of type $0$ 
				not fixed by the involution and such that the partial normalization at $\{p, \iota(p)\}$ yields a
				 curve of genus $j$ and a curve of genus $g-1-j$. The constant $\alpha(g)$ is to be 
				 interpreted as $0$ when $g=3,4$ and as $\frac{2g-1}{3}$ when $g \geq 5$. (This follows
				 from Yamaki's formulas upon noting that $\delta_0 - \xi_0 = \sum_{j=1}^{[(g-1)/2]} \xi_j$, 
				 by definition.)
				 
			\item (Yamaki, \cite{Yamaki_Bogomolov_Genus_3_2002}) If the genus of $C$ is $3$ and 
				$C$ is not hyperelliptic, then 
				\[
					\inf_{D \in \Div^1\left(\alg{C}\right)}a'(D) \geq  \frac{1}{792} \delta_0
					+ \frac{1}{6} \delta_1.
				\]
				
			\item (Gubler, \cite{Gubler_Bogomolov_2007}) If the Jacobian of $C$ has totally 
				degenerate reduction over $y \in Y(k)$, then the Geometric Bogomolov Conjecture holds 
				for $C$. This hypothesis is equivalent to saying the first Betti number of the reduction graph 
				equals $g$, the genus of $C$.\footnote{
				This equivalence does not seem to be written anywhere in the literature, so
				we give a brief sketch of its proof at the request of the referee. The genus of $C$ agrees with 
				the genus of its reduction graph 
				at $y$, which is the first Betti number of the graph {\it plus} the sum of the 
				geometric genera of the irreducible components of the fiber of $X$ over $y$. (See
				\S\ref{Sec: Main Definitions} and \S\ref{Sec: Red to Graphs}, especially 
				Proposition~\ref{Prop: pm-graph}.) To show that the first Betti
				number of the reduction graph equals the genus of $C$ is then equivalent to asserting all of the 
				irreducible components of the fiber of $X$ over $y$ are rational. But
				to say that the Jacobian of $C$ has totally degenerate reduction over $y$ is equivalent to saying 
				the abelian rank of the special fiber of the identity component of its N\'eron model is zero, which 
				in turn is equivalent to saying that every irreducible component of the fiber of $X$ over $y$ is 
				rational. These facts about the Jacobian can be deduced 
				from \cite[expos\'e $n^{\circ}$ 1, p.28]{Asterisque_86}.
				}
		\end{enumerate}
		
	For genus $2$ curves, we obtain the same coefficient on $\delta_0$ as Moriwaki. 
	This can be explained by the fact that every 
	curve of genus $2$ is hyperelliptic, and one of the key inequalities used to prove our result becomes an 
	equality for hyperelliptic curves. (It is the second half of inequality \eqref{Eqn: Main Inequalities}.) For 
	genus $3$ curves, we obtain a lower bound of $\frac{1}{567}\delta_0 + \frac{2}{21}\delta_1$, 
	which is in general neither stronger nor weaker than Yamaki's result. One can recover the result (c) above using
	the computation in \cite[Prop.~4.4.3]{Zhang_Gross-Schoen_Cycles_2008} and the method of this paper.
	
\begin{remark}
	In all but one of the related articles of Moriwaki and both of the articles of Yamaki, they use a slightly different 
	embedding of the curve into its Jacobian and a different measure of finiteness than we do. At first glance, the
	results in their papers will look quite different than as stated here.
\end{remark}

\subsection{Idea of the Proof}

	Now we sketch the proof of Theorem~\ref{Thm: Main Theorem} in order to motivate the layout of the paper. The detailed proof will be given in \S\ref{Sec: Red to Graphs}. 
	
	Assume for simplicity that $C$ has semistable reduction over $K$, and let $f: X \to Y$ be the minimal regular model of $C$. If $f: X \to Y$ is smooth, then the result is essentially due to Par{\v{s}}in, so we may assume that $f$ has at least one singular fiber. We use the following inequalities proved by Zhang:
	\be \label{Eqn: Main Inequalities}
		a'(D) \geq \frac{(\omega_a, \omega_a)}{4(g-1)} 
			\geq \frac{1}{2(2g+1)} \sum_{y \in Y(k)} \Zphi{\agam_y},
	\ee
where $\omega_a$ is the canonical sheaf of $C$ equipped with the admissible adelic metric in the sense of Zhang, $(\omega_a, \omega_a)$ is the admissible pairing of $\omega_a$, and $\Zphi{\agam_y}$ is a graph invariant associated to the reduction graph $\agam_y$ of the fiber $f^{-1}(y)$. Recall that the reduction graph has vertices in bijection with the irreducible components of $f^{-1}(y)$ and a segment of length~1 between two vertices for every node shared by the  two corresponding irreducible components. If an irreducible component of $f^{-1}(y)$ has nodal self-intersections, then the corresponding vertex of the reduction graph is joined to itself by loops of length~1. For example, $\agam_y$ is a point if and only if $f$ is smooth over $y$, and in this case $\Zphi{\agam_y}= 0$. As $f$ is generically smooth, the above sum over points of $Y$ is finite. To prove Theorem~\ref{Thm: Main Theorem}, it suffices to produce lower bounds for $\Zphi{\agam_y}$ for all possible (polarized) metric graphs that can arise from fibers of $f$. We state such a lower bound
as Theorem~\ref{Thm: Graph Bounds} below.

	The majority of this paper is devoted to exhibiting lower bounds for the invariants $\Zphi{\agam_y}$. In \S\ref{Sec: Main Definitions}--\ref{Sec: Invariants} we gather the definitions and basic facts on metric graphs used to state and prove such bounds. In \S\ref{Sec: Red to Graphs} we give a more complete treatment of the argument sketched above in order to pass from a statement about algebraic curves to a statement about graph invariants. In \S\ref{Sec: Rapid} we provide explicit formulas for computing $\Zphi{\agam}$ that allow one to implement its calculation in a computer algebra package. In \S\ref{Sec: Continuity}--\S\ref{Sec: Strategy} we reduce the proof of the theorem to a handful of explicit calculations. These calculations are immensely complicated, albeit not very difficult. We perform them in \textit{Mathematica} \cite{Mathematica_6.0.1.0} and summarize the conclusions in \S\ref{Sec: Data}. 
\begin{withcode} Finally, we reproduce the computer code and \textit{Mathematica} notebooks used for the computations. \end{withcode}
\begin{withoutcode} The \textit{Mathematica} code and notebooks have been reproduced at the end of the arXiv edition of this article \cite{Faber_EGBC_arXiv_2008}. \end{withoutcode}

	\textbf{Acknowledgments.} My teacher, Shou-wu Zhang, deserves my warmest thanks for suggesting this project and for having an open door during its progression. Thanks also go to Xinyi Yuan for showing me the inequalities used in the proof of Proposition~\ref{Prop: Verify}, to Matt Baker for helpful feedback on the exposition, and to the anonymous referee for suggesting various technical and stylistic corrections. This research was partially supported by the NSF through Shou-wu Zhang's grant DMS-070322.


\section{Polarized Metric Graphs and Their Invariants}
\label{Sec: Definitions}

\subsection{Polarized Metric Graphs}
\label{Sec: Main Definitions}

	In order to fix terminology, we give a very basic introduction to polarized metric graphs. For a more complete discussion of this topic, see \cite{Baker_Faber_2006, Baker_Rumely_HAOMG_2007}. A \textbf{metric graph} $\Gamma$ is a compact connected metric space for which each point $p$ admits a neighborhood isometric to one of the form
	\[
		U_{r,v} = \{te^{2\pi ik/v} \in \CC: 0 \leq t < r, k=1, \ldots, v\},
	\]
for some positive real number $r$ and positive integer $v$, where we endow $U_{r,v}$ with the path metric. The number $v=v(p)$ is called the \textbf{valence} of the point $p$. For convenience, we also allow the topological space consisting of a single point to be called a metric graph; in this case, $v(p) = 0$. 

	When working with metric graphs, it is often easier to ``discretize'' and work with classical combinatorial objects. The following two definitions are slightly non-standard, but they will simplify our discussion a great deal. A \textbf{combinatorial graph} is a connected multigraph; i.e., a connected graph, possibly with multiple or loop edges. If $G$ is a combinatorial graph, we write $V(G)$ and $E(G)$ for the vertex and edge sets, respectively. A \textbf{weighted graph} is a combinatorial graph along with a function $\ell: E(G) \to \RR_{>0}$. We call $\ell(e)$ the \textbf{length} of the edge $e$. 
	
	Any weighted graph $G$ gives rise to a metric graph $\Gamma$ first by building a $1$-dimensional CW-complex using the vertex and edge sets as gluing data. An edge $e$ of $G$ induces a \textbf{segment} $e \subset \Gamma$, which we equip with the Lebesque measure of total mass $\ell(e)$. The metric on $\Gamma$ is defined to be the induced path metric. Conversely, given a metric graph $\Gamma$, one obtains a weighted graph $G$ as follows. Let $V(G)$ be any nonempty finite subset of $\Gamma$ containing all points of valence different from $2$. Two vertices $p_1$ and $p_2$ of $G$ are joined by an edge if there is a path in $\Gamma$ from $p_1$ to $p_2$ that does not pass through any other vertices. Any weighted graph constructed in this way from $\Gamma$ is called a \textbf{model} of $\Gamma$.
	
	For any model $G$ of a metric graph $\Gamma$, we have the equality
	\be \label{Eq: Betti number}
		\dim_{\RR} H^1(\Gamma, \RR) = \#E(G) -\#V(G) + 1.
	\ee
This formula can be proved by counting the edges in a spanning tree for $G$. The number $b_1 = \dim_{\RR} H^1(\Gamma, \RR)$ is called the \textbf{first Betti number} of $\Gamma$. 

	A metric graph is \textbf{irreducible} if it cannot be disconnected (as a topological space) by deleting any single point. A metric graph $\Gamma$ is \textbf{cubic} if the following two conditions hold:
		\bi
			\item There exists $p \in \Gamma$ with valence $3$.
			
			\item Every $p \in \Gamma$ has valence equal to $2$ or $3$. 
		\ei
This definition implies that a cubic metric graph $\Gamma$ admits a $3$-regular model in the sense of combinatorial graph theory. Indeed, take the vertex set $\{p \in \Gamma: v(p) =3\}$. (Recall that a combinatorial graph is $3$-regular if every vertex has valence~3.)

	Let $\Gamma$ be a metric graph and choose a model $G$. Each segment $e$ is isometric to either a closed interval or a circle of length $\ell(e)$. If $e$ is isometric to a segment, a choice of isometry $e \simarrow [0, \ell(e)]$ gives an orientation on $e$. We write $e^-$ and $e^+$ for the vertices corresponding to $0$ and $\ell(e)$, respectively, under the chosen isometry. If $e$ is instead isometric to a circle, we define $e^- = e^+$ to be the unique marked vertex on $e$. We define the \textbf{total length} of $\Gamma$ to be $\ell(\Gamma) = \sum_e \ell(e)$. Evidently it does not depend on the choice of model. 
	
	A \textbf{polarized metric graph}, or \textbf{pm-graph}, consists of the data of a metric graph $\Gamma$ and a vertex weight function $q: \Gamma \to \ZZ$ with the following properties:
	\bi
		\item $q$ has finite support in $\Gamma$; 
		\item $q(x) \geq 0$ for all $x \in \Gamma$;
		\item The \textbf{canonical divisor} $K := \sum_{p \in \Gamma} \left( 2q(p) +v(p) - 2 \right). [p]$
			is effective.
	\ei
By way of notation, we will write $\agam = (\Gamma, q)$ for a pm-graph. If $\Gamma = \{p\}$ is a single point, then $K$ is effective if and only if $q(p) \geq 1$. Otherwise, it suffices to focus only on points of valence~$1$ to check that $K$ is effective. 

	The \textbf{genus} of $\agam$ is defined to be
	\be \label{Eqn: Genus}
		\Zg{\agam} = \dim_{\RR} H^1(\Gamma, \RR) + \sum_{p \in \Gamma} q(p).
	\ee
We will abbreviate by $g = \Zg{\agam}$ when there is no possibility of confusion.  It is easy to see that the degree of the canonical divisor is  $\deg(K) = 2g - 2$.

	Let $\agam = (\Gamma, q)$ be a pm-graph. A \textbf{model} of $\agam$ is a model of the metric graph $\Gamma$ that contains the support of $q$ in its vertex set. We will call $\agam$ \textbf{irreducible} (resp. \textbf{cubic}) if its underlying metric graph has this property. 
	
	With $\agam$ as in the last paragraph, choose a point $p \in \Gamma$ of valence~$2$ outside the support of the function $q$. All but finitely many points of $\Gamma$ have these two properties. We wish to define the \textbf{type} of the point $p$ to be an integer in the interval $[0, g/2]$. The type of $p$ will determine how it contributes to the invariant $\Zphi{\agam}$. (See Lemma~\ref{Lem: Decomp}.) We say $p$ is of \textbf{type~0} if $\Gamma \smallsetminus \{p\}$ is connected. Otherwise we may write $\Gamma = \Gamma_1 \cup \Gamma_2$ for some metric subgraphs $\Gamma_1$ and $\Gamma_2$ such that $\Gamma_1 \cap \Gamma_2 = \{p\}$. While the data $ (\Gamma_j, q|_{\Gamma_j})$ does not define a pm-graph, the definition of genus~\eqref{Eqn: Genus}  still makes sense. Moreover, one sees that
	\[
		g = \Zg{\agam} = g\left(\Gamma_1, q|_{\Gamma_1}\right) +g\left(\Gamma_2, q|_{\Gamma_2}\right), 
	\]
so that one of these two integers lies between 1 and $g/2$, inclusive. Define $p$ to be a point of \textbf{type $i$}, where $i$ is the minimum of $g\left(\Gamma_1, q|_{\Gamma_1}\right)$ and $g\left(\Gamma_2, q|_{\Gamma_2}\right)$. See Figure~\ref{Fig: Types} for an example.

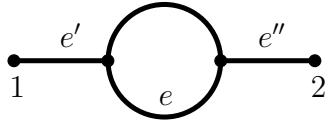
\begin{figure}[htb]
	\begin{picture}(100,48)(20,-5)

		\put(0,0){\includegraphics{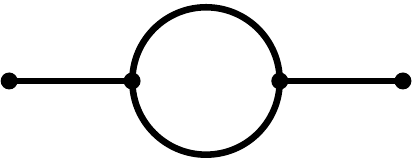}}
		\put(1,9){$1$}
		\put(115,9){$2$}
		\put(20, 28){$e'$}
		\put(57.5, 6){$e$}
		\put(95, 28){$e''$}
		
	\end{picture}
	
	\caption{Here we have a pm-graph $\agam = (\Gamma, q)$ of genus~4. The indices $1$ and $2$ are
	the values of the function $q$ at those vertices, and we suppose $q$ is zero at all other points. A point $p$ in 
	the interior of segment $e$ has type~0 as $\Gamma \smallsetminus \{p\}$ is connected. If $p$ lies in the interior
	of $e'$ (resp. $e''$), then it has type~1 (resp. type~2) since it splits the graph into a piece of genus~1 and a
	piece of genus~3 (resp. two pieces of genus~2).}
	\label{Fig: Types}
\end{figure}

	Finally, define $\ell_i\left(\agam\right)$ to be the total length of all points of type $i$ (in the sense of Lebesgue measure). Evidently $\ell(\Gamma) = \sum_i \ell_i\left(\agam\right)$. (Note that $\ell_i\left(\agam\right)$ for $i > 0$ is an invariant of the pm-graph $\agam$, while $\ell_0\left(\agam\right)$ and $\ell(\Gamma)$ are metric invariants.)
	
	
\subsection{Admissible Measures}
\label{Sec: Admissible Stuff}

	For this section, fix a pm-graph $\agam = (\Gamma, q)$. We now give a very brief description of the theory of admissible measures; for more complete references see \cite{Baker_Rumely_HAOMG_2007, Zhang_Admissible_Pairing_1993}. 
	
	There is a measure-valued Laplace operator on $\Gamma$. Let $f$ be a continuous function that is $C^2$ outside of some vertex set for $\Gamma$, and such that $f''$ is square-integrable against the Lebesgue measure on $\Gamma$. The \textbf{Laplacian} of $f$ is defined as
	\[
		\Delta_x(f) = -f''(x) dx - \sum_{p \in \Gamma} \sigma_p(f) \delta_p(x),
	\]
where $dx$ is the Lebesgue measure on $\Gamma$, $\delta_p$ is the point mass at $p$, and $\sigma_p(f)$ is the sum over all tangent directions at $p$ of the outward pointing derivatives of $f$. One of the key properties we will use is that $\Delta_x(f) = 0$ if and only if $f$ is a constant function. 

	 For any signed Borel measure $\nu$ on $\Gamma$ of total mass~$1$, the \textbf{Green's function} with respect to $\nu$ is the unique function $g_{\nu}: \Gamma \times \Gamma \to \RR$ satisfying the following three properties:
	\begin{enumerate}
		\item $g_{\nu}(x,y)$ is continuous on $\Gamma \times \Gamma$ and symmetric in $x$ and $y$.
		\item  $\Delta_x g_{\nu}(x,y) = \delta_y - \nu$ for each $y \in \Gamma$.
		\item  $\int_\Gamma g_\nu(x,y) d\nu(x) = 0$ for each $y \in \Gamma$.
	\end{enumerate}

	Here and in what follows it will be convenient to write $f(D) = \sum_{p \in \Gamma} n_p f(p)$ if $f: \Gamma \to \RR$ is a function and $D = \sum_{p \in \Gamma} n_p [p]$ is a divisor on $\Gamma$. As is customary for divisors, we suppose $n_p=0$ for all but finitely many $p \in \Gamma$.

	There exist measures whose Green's functions satisfy extra properties. For any divisor $D$ on $\Gamma$ with $\deg(D) \not= -2$, there exists a unique Borel measure $\mu_D$ of total mass~1 and a unique constant $a = a\left(\agam, D\right)$ such that
	\be \label{Eqn: Admissible Constants}
		g_{\mu_D}(x,x) + g_{\mu_D}(x, D) = a, \quad x \in \Gamma.
	\ee
The measure $\mu_D$ is called the \textbf{admissible measure} with respect to the divisor $D$. Of particular interest to us is the case $D=K$, where $K$ is the canonical divisor on $\agam$. Note that $\deg K \not= -2$ as $K$ is effective for any pm-graph. We will write $\mu = \mu_K$ for the duration of this article.
		
	The \textbf{effective resistance function} $r(x,y)$ --- familiar from circuit theory ---  allows us to give an explicit formula for the measure $\mu_D$. By definition, $r(y,z) = j_z(y,y)$, where $f(x) = j_z(x,y)$ is the fundamental solution to the Laplace equation 
	\[
		\Delta_x f = \delta_y - \delta_z, \quad f(z)=0. 
	\]
Let $D$ be a divisor of degree different from $-2$ and suppose $G$ is a model of $\Gamma$ containing the support of the divisor $D$. Let $r(e)$ be the effective resistance between the endpoints of the segment $e$, and let $F(e) = 1 - r(e)/\ell(e)$. Then we have
		\benn
			\mu_D = \frac{1}{\deg(D) + 2} \left\{ \delta_D - \sum_{p \in V(G)}\left( v(p) - 2\right) + 2 
				\sum_{e \in E(G)} F(e) \frac{dx|_e}{\ell(e)} \right\},
		\eenn
In particular, if $g = \Zg{\agam} \geq 1$ is the genus of $\agam$ and $G$ is a model of $\agam$, then 
	\benn
		\mu = \mu_K = \frac{1}{g}\left\{ \sum_{p \in \Gamma} q(p)\delta_p + \sum_{e \in E(G)} 
				F(e) \frac{dx|_e}{\ell(e)}\right\}.
	\eenn


\subsection{The Invariant $\Zphi{\agam}$}
\label{Sec: Invariants}

	Define three invariants associated to a pm-graph $\agam$ of genus $g \geq 1$:
	\begin{eqnarray*}
			\Zepsilon{\agam} & = & \dint r(x,y) \delta_K(x) d\mu(y) \\
			\Ztau{\agam} & = &  \frac{1}{2}\dint r(x,y) d\mu(x) d\mu(y) \\
			\Zphi{\agam} & = & 3 g \Ztau{\agam}
				- \frac{1}{4}\left(\Zepsilon{\agam} + \ell(\Gamma) \right).
	\end{eqnarray*}

\begin{remark}
	We have chosen to use the notation $\Ztau{\agam}$ instead of $\tau\left(\agam\right)$ as is favored in 
	\cite{Zhang_Gross-Schoen_Cycles_2008}. This is to avoid confusion with the closely related ``Tau-invariant'' 
	that appears in \cite{Baker_Rumely_HAOMG_2007, Cinkir_Thesis_2007}. One can show that 
	$\Ztau{\agam}$ agrees with the constant $a = a(\agam, K)$ appearing in \eqref{Eqn: Admissible Constants}. 
	Also, the invariant $\Zepsilon{\agam}$ is the same as the invariant $e_y$ appearing in
	 \cite{Moriwaki_Bogomolov_Genus_2_1996} and the invariant $\epsilon(G,D)$ appearing
	in \cite{Moriwaki_Bogomolov_Stable_Trees_1998, Yamaki_Bogomolov_Hyperelliptic_2008, 
	Yamaki_Bogomolov_Genus_3_2002}, although it requires a small computation to see it.
\end{remark}

\begin{remark}
	An important --- albeit trivial --- example is the case $\agam = (\Gamma, q)$ in which $\Gamma = \{p\}$ consists 
	of a single point. Then $\agam$ is a pm-graph if and only if $q(p) \geq 1$. It follows easily that all three of the
	 associated invariants $\Zepsilon{\agam}$, $\Ztau{\agam}$, and $\Zphi{\agam}$ are zero.
\end{remark}

	We now show that $\Ztau{\agam}$ may be removed from the definition of $\Zphi{\agam}$.
	
\begin{prop} \label{Prop: Tau in terms of epsilon}
	For any pm-graph $\agam$ of genus $g$, we have
		\benn
			\Ztau{\agam} = \frac{2g-1}{4g(g-1)} \Zepsilon{\agam} - \frac{r(K,K)}{8g(g-1)}.
		\eenn
\end{prop}

\begin{proof}
	Let $g_0: \Gamma \times \Gamma \to \RR$ and $\mu_0$ be the admissible Green's function and the admissible measure associated to the divisor $D = 0$.\footnote{The object $\mu_0$ is called the \textbf{canonical measure} by Baker and Rumely \cite[\S14]{Baker_Rumely_HAOMG_2007}.} Recall that this means 
	\begin{enumerate}
		\item $g_0(x,y)$ is continuous on $\Gamma \times \Gamma$ and symmetric in $x$ and $y$.
		\item $\Delta_x g_0(x,y) = \delta_y - \mu_0$ for each $y \in \Gamma$.
		\item $\int_\Gamma g_0(x,y) d\mu_0(x) = 0$ for each $y \in \Gamma$.
		\item $g_0(x,x)$ is a constant independent of $x$.
	\end{enumerate}
Recall also that $\mu$ (with no subscript) is the admissible measure for the divisor $D = K$. By \cite[Thm.~14.1]{Baker_Rumely_HAOMG_2007} and the discussion in \S\ref{Sec: Admissible Stuff}, we have the following further properties of $\mu_0$ and $g_0(x,y)$:
	\begin{eqnarray}
		g_0(x,y) &=& - \frac{1}{2}r(x,y) + C, \text{ for some constant $C$}, \label{Eqn: g_0} \\
		\mu_0 &=& \frac{1}{2}\Delta_x \left(r(x,y)\right) + \delta_y, 
			\text{ for any $y \in \Gamma$}, \label{Eqn: mu_0}\\
		\mu & = & \frac{1}{2g}\left(2\mu_0 + \delta_K\right). \label{Eqn: mu vs mu_0}
	\end{eqnarray}

	Now define a new function $f: \Gamma \to \RR$ by  
	\benn
		f(x) = \frac{1}{2}\int_\Gamma r(x,y)d\mu(y) - \frac{1}{4g} r(x,K).	
	\eenn
We may rewrite $f$ using \eqref{Eqn: g_0} and \eqref{Eqn: mu vs mu_0} to get
	\benn
		\ba
			f(x) &= - \frac{1}{2g} \int g_0(x,y) \left(2\mu_0(y) + \delta_K(y)\right) + C - \frac{1}{4g}r(x,K) \\
				&= -\frac{1}{2g}g_0(x,K) + C - \frac{1}{4g} r(x,K), \qquad \text{by property~(c)}.
		\ea
	\eenn
Taking the Laplacian of both sides with respect to $x$ and applying property~(b) and equation~\eqref{Eqn: mu_0} gives
	\benn
		\Delta_x f = 
		-\frac{1}{2g}\left(\delta_K - (2g-2)\mu_0\right) - \frac{1}{4g}\left(2(2g-2)\mu_0 - 2\delta_K\right) = 0.
	\eenn
	
	This implies that $f$ is a constant, say $f \equiv A$. Integrating $f$ against $\mu$ gives
	\benn
		A = \int f(x) d\mu(x)  = \Ztau{\agam} - \frac{1}{4g}\Zepsilon{\agam}.
	\eenn
On the other hand, integrating against $\delta_K$ yields
	\benn
		(2g-2)A = f(K)  = \frac{1}{2}\Zepsilon{\agam} - \frac{1}{4g}r(K,K).
	\eenn
Solving these last two equations for $\Ztau{\agam}$ completes the proof.
\end{proof}

	Combining Proposition~\ref{Prop: Tau in terms of epsilon} with the definition of $\Zphi{\agam}$ immediately gives 
	
\begin{cor} \label{Cor: No tau}
	For any pm-graph $\agam$ of genus $g$, we have
		\benn
			\Zphi{\agam}  =  \frac{5g-2}{4(g-1)} \Zepsilon{\agam} 
				 - \frac{3}{8(g-1)}r(K,K) - \frac{1}{4}\ell(\Gamma).
		\eenn
\end{cor}


\section{Reducing the Proof of Theorem~\ref{Thm: Main Theorem} to Graph Theory}
\label{Sec: Red to Graphs}

	In this section we give the argument that allows us to pass from a statement about algebraic curves to one about polarized metric graphs. While the argument is not difficult, it is new to the literature and relies on a recent theorem of Zhang. As such we give a detailed description. The reader may want to look at the beginning of \S\ref{Sec: Statement} for references on semistable reduction and at \cite[pp.9-10]{Chinburg_Rumely_Capacity_Pairing_1993} for further details on the theory of reduction graphs.
	
\subsection{Reduction Graphs}
	
	Suppose $k$ is an algebraically closed field of characteristic zero and $K$ is the function field of a smooth curve over $k$. Let $Y$ be the proper smooth $k$-curve with function field $K$. 	The points $Y(k)$ are in bijective correspondence with places of $K$ --- i.e., discrete valuations of the function field $K$. For a point $y \in Y(k)$, let us write $K_y$ for the completion of $K$ with respect to the corresponding discrete valuation.

	Let $C$ be a proper smooth geometrically connected curve over $K$ of genus $g \geq 2$. By the Semistable Reduction Theorem there exists a finite extension $E / K_y$ such that $C_E = C \times_K E$ has semistable reduction. That is, if $\OO_E$ is the valuation ring of $E$, then there is a flat proper $\OO_E$-scheme with generic fiber $C_E$ and with semistable special fiber. Let us take $f: \mathcal{C}_E \to \Spec \OO_E$ to be the minimal regular (semistable) model of $C_E$. 	

	In what follows we identify $y$ with the closed point of $\Spec \OO_E$. Define a combinatorial graph $G_y$ with vertex set in bijection with the set of irreducible components of $f^{-1}(y)$, and with an edge between vertices $p, p'$ for each node shared by the corresponding irreducible components $Z, Z'$. In particular, a vertex of $G_y$ admits a loop edge for each singularity of the corresponding irreducible component of $f^{-1}(y)$. Endow $G_y$ with the structure of a weighted graph by defining $\ell(e) = 1 / [E: K_y]$ for each edge $e$. (As $k$ is algebraically closed, $[E:K_y]$ is the ramification index of the extension.) Define $\Gamma_y$ to be the metric graph induced by the weighted graph $G_y$. Define $q: \Gamma_y \to \ZZ$ as follows. Set $q(p) = 0$ for all $p \not\in V(G_y)$. For any vertex $p$ of $G_y$ define $q(p)$ to be the geometric genus of the irreducible component of $f^{-1}(y)$ corresponding to the vertex $p$. Finally, set $\agam_y = (\Gamma_y, q)$. The key point of this construction is that, by virtue of the multiplicativity of ramification indices, the object $\agam_y$ is independent of the choice of extension $E / K_y$. This follows for example from \cite[Fact~2.5]{Chinburg_Rumely_Capacity_Pairing_1993}.\footnote{As the field of constants $k$ is algebraically closed in our setting, the approach of Chinburg / Rumely in \cite{Chinburg_Rumely_Capacity_Pairing_1993} can proceed without their notion of ``well-adjusted'' regular models. They require this extra technical tool in order to work with more general global fields.}

	For example, if $f^{-1}(y)$ is smooth, then $\Gamma_y = \{p\}$ is a point, and $q(p) = g$, the genus of $C$. Conversely, if $\Gamma_y$ consists of a single point, then $f^{-1}(y)$ is smooth.

\begin{prop} \label{Prop: pm-graph}
	For each $y \in Y(k)$, the reduction graph $\agam_y = (\Gamma_y, q_y)$ is a pm-graph of genus $g$, 
	where $g$ is the genus of the curve $C$.
\end{prop}

\begin{proof}
	To check that $\agam_y$ is a pm-graph, we only need to verify that its canonical divisor is effective. To avoide confusion, we will write $D_y$ for the canonical divisor of $\agam_y$ (the symbols $K$ and $K_y$ being already in use). Recall that for $p \in \Gamma_y$ the order of $D_y$ at $p$ is given by 
	\[
		\ord_p(D_y) = 2q(p) +v(p) - 2,	
	\]
where $v(p)$ is the valence of the point $p$. Let $p_1, \ldots, p_n$ be the points of $\Gamma_y$ corresponding to irreducible components $Z_1, \ldots, Z_n$ of the fiber $f^{-1}(y)$, numbered accordingly. It is clear from the definitions that $\ord_p(D_y) = 0$ for any $p \not= p_i$. The remaining cases can be calculated via intersection theory on the surface $\mathcal{C}_E$. If $\omega_f = \omega_{\mathcal{C}_E / \OO_E}$ is the relative dualizing sheaf, then we claim
	\be \label{Eq: canonical intersection}
		\ord_{p_i}(D_y) = Z_i. \omega_f.
	\ee
By the work of Arakelov, the sheaf $\omega_f$ is numerically effective \cite[Cor.~9.3.26]{Qing_Liu_Algebraic_Geometry}. So for any irreducible curve $Z_i \subset f^{-1}(y)$ we have $Z_i. \omega_f \geq 0$, which shows the effectivity of $D_y$. 

	We now prove~\eqref{Eq: canonical intersection}. Without loss of generality we may suppose $i=1$, and we write $p = p_1$ and $Z = Z_1$. Also, let $ G_y$ be the combinatorial graph dual to the fiber $f^{-1}(y)$ as constructed above. We can compute the arithmetic genus of $Z$ by the adjunction formula for regular fibered surfaces \cite[Thm.~9.1.37]{Qing_Liu_Algebraic_Geometry}:
	\be \label{Eq: adjunction}
		\ba
		p_a(Z) &= 1 + \frac{1}{2}\left( Z^2 + Z.  \omega_f \right) \\
		\Longrightarrow Z. \omega_f &= 2p_a(Z) - Z^2 - 2.
		\ea
	\ee
The blowing-up formula relating arithmetic and geometric genus for a curve on a fibered surface behaves especially nicely in this case due to the presence of only nodal singularities. We have
	\benn
		\ba
			p_a(Z) &= p_g(Z) + \#\{ \text{nodes of $Z$} \} \\
				&= q(p) + \#\{ \text{loop edges of $G_y$ at $p$}\}.
		\ea
	\eenn
Define $v_{l}(p)$ (resp. $v_n(p)$) to be the valence of $G_y$ at $p$ contributed by loop edges (resp. non-loop edges). So $v(p) = v_l(p) + v_n(p)$. Each loop contributes $2$ to the valence at $p$, so 
	\be \label{Eq: Blow-up}
		p_a(Z) = q(p) + \frac{1}{2}v_l(p).	
	\ee

	Next note that $Z \sim - \sum_{i > 1} Z_i$ since $[f^{-1}(y)] = \sum_i Z_i$ is a principle divisor. The intersections of the $Z_i$'s with $Z$ for $i > 1$ are in bijective correspondence with the non-loop edges of $G_y$ at $p$, so	
	\be \label{Eq: Z squared}
			Z^2 = - \sum_{i > 1} Z_i . Z = - \#\{ \text{non-loop edges at $p$} \} = - v_n(p).
	\ee
Finally, we combine~\eqref{Eq: adjunction}, \eqref{Eq: Blow-up}, and~\eqref{Eq: Z squared} to find
	\[
		Z.\omega_f = 2\left(q(p) + \frac{1}{2}v_l(p) \right) + v_n(p) - 2 
		= 2q(p) +v(p) - 2 = \ord_p(D_y),
	\]
which is what we wanted.

	To compute the genus of $\agam_y$, we calculate the degree of its canonical divisor: 
		\benn
			\ba
				2\Zg{\agam_y} - 2 =  \sum_{p \in \Gamma_y} \ord_p(D_y) 
					&= \sum_i \ord_{p_i}(D_y) \\
					&= \sum_{i}  Z_i. \omega_f \\
					&= [f^{-1}(y)].\omega_f \\
					&= \deg(\omega_{C_{E}})  = 2g-2,
			\ea
		\eenn
where $\omega_{C_{E}}$ is the canonical sheaf on $C_{E}$. Hence the genus of $\agam_y$ is $g$.
\end{proof}

\subsection{The Work of Zhang}

	We continue to assume that $C$ is a proper smooth connected curve of genus $g \geq 2$ over the function field $K$ with field of constants $k$. 	Write $\alg{C} = C \times_K \alg{K}$. We begin with the following result of Zhang:
	
\begin{thm}[Zhang, {\cite[Thm.~5.6]{Zhang_Admissible_Pairing_1993}}] \label{Thm: Admiss Pairing}
	For any divisor $D$ on $\alg{C}$ of degree~1,
		\[
			a'(D) \geq \frac{(\omega_a, \omega_a)}{4(g-1)} 
				+ \left(1 - \frac{1}{g}\right)\hat{h}\left(D - \frac{K_C}{2g-2}\right),
		\]
	where $\omega_a$ is the admissible relative dualizing sheaf associated to $C$ and $K_C$ is a canonical divisor on $C$. 
\end{thm}

	We want to relate the admissible intersection number $(\omega_a, \omega_a)$ to invariants of a global semistable model of $C$. Choose a finite extension $K' / K$ such that $C_{K'} = C \times_K K'$ has semistable reduction, and write $d = [K':K]$. Let $Y'$ be the proper smooth curve over $k$ with function field $K'$. Define $f:X' \to Y'$ to be the minimal regular (semistable) model of $C_{K'}$. Notice that we are working with a global semistable model of $C_{K'}$ in contrast to the local model used in the previous section. Let $\omega_{X' / Y'}$ be the relative dualizing sheaf. 

	When $f: X' \to Y'$ is smooth and $C$ is not a constant curve, it turns out that
	\[
		(\omega_a, \omega_a) = \frac{\omega_{X' / Y'} . \omega_{X' / Y'}}{d} \geq \frac{12}{d},
	\]
by a result of Par{\v{s}}in \cite{Parsin_Algebraic_Curves_FF_1968}. Then Zhang's result and positivity of the canonical height implies
	\[
		\inf_{D \in \Div^1\left(\alg{C}\right)} a'(D) \geq \frac{(\omega_a, \omega_a)}{4(g-1)} 
			\geq \frac{3}{d(g-1)},
	\]
which proves Theorem~\ref{Thm: Main Theorem} when $f$ is smooth.

	The rest of the discussion applies whether or not $f$ is smooth, but it will be of greatest interest when $f$ has some singular fibers. To relate the quantities $a'(D)$ to our graph invariants in this case, we use the following recent result:

\begin{thm}[Zhang, {\cite[Cor.~1.3.2 and \S1.4]{Zhang_Gross-Schoen_Cycles_2008}}]
	\label{Thm: Gross-Schoen}
	With the notation above, 
	\[
		(\omega_a, \omega_a) \geq \frac{2g-2}{2g+1} \sum_{y \in Y(k)} \Zphi{\agam_y}.
	\]
\end{thm}

	Recall from \S\ref{Sec: Invariants} that $\Zphi{\agam_y} = 0$ whenever $\agam_y$ consists of a single point. As $C$ is smooth over $K$, there exists an affine subscheme $U \subset Y$ and a proper smooth morphism $X \to U$ with generic fiber $C$. Hence $\agam_y$ is a single point for all $y \in U(k)$, which means the above sum over points of $Y$ is actually finite. 

	Continuing with the notation above, let $\alpha: Y' \to Y$ be a morphism of proper smooth curves realizing the extension of function fields $K'/K$. For $y' \in Y'(k)$, we write $\agam_{y'}$ for the reduction graph of $C_{K'}$ at the point $y'$. We need to relate $\Zphi{\agam_{y'}}$ and $\Zphi{\agam_{y}}$ whenever $\alpha(y') = y$. In fact, we claim
	\be \label{Eq: Ramification Claim}
		\Zphi{\agam_{y'}} = e_{y'} \Zphi{\agam_y},
	\ee	
where $e_{y'}$ is the ramification index of $\alpha$ at $y'$. The point is that $\agam_{y'}$ is defined relative to the base field $K'$, so the lengths of its edges will differ from those of $\agam_y$ by exactly the ramification index. 

	More precisely, let $E = K'_{y'}$ be the completion of $K'$ with respect to the discrete valuation corresponding to $y'$. Then $X' \times_{Y'} \Spec \OO_{E}$ is the minimal regular (semistable) model of $C_{E}= C_{K'} \times_{K'} E$ because the fiber of $X'$ over $y'$ is unaffected by base change to the completed local ring $\OO_E$. Thus we can define $\agam_{y'}$ using this model. Notice that all of its segments will have length~1 because we did not need to make a finite extension of $E = K'_{y'}$ in order to obtain a semistable model over $K'$. On the other hand, $K_y \subset E$ since $y'$ lies above $y$, and so we can define $\agam_y$ by passing to the extension $E / K_y$. The degree $[E: K_y]$ is equal to the ramification index $e_{y'}$ of $\alpha: Y' \to Y$ at $y'$. The underlying combinatorial graph structure of $\agam_y$ will be identical to that of $\agam_{y'}$, but now a segment of $\agam_{y}$ will have length $1 / [E:K_y] = 1/e_y$. One can see from the proof of Proposition~\ref{Prop: Special Forms} that if a pm-graph $\agam'$ is obtained from $\agam$ by scaling all segments by the same quantity $\lambda$, then $\Zphi{\agam'} = \lambda\Zphi{\agam}$. The claim \eqref{Eq: Ramification Claim} follows.

	Next we need to know how the lengths $\ell_i\left(\agam_{y'}\right)$ relate to the singular indices $\delta_i = \delta_i(X'/Y')$. These notions are essentially dual to each other: a node of type $i$ in the fiber $f^{-1}(y')$ corresponds to a segment $e$ of length $1$ in $\Gamma_{y'}$. Each point of the segment $e$ (aside from the endpoints) is a point of $\Gamma_{y'}$ of type $i$. This correspondence allows one to verify that
		\be \label{Eq: Singular lengths}
			\sum_{y' \in Y'(k)} \ \ell_i\left(\agam_{y'}\right) = \delta_i(X'/Y'), 
				\qquad 0 \leq i \leq \left[g/2\right].
		\ee

	We are now ready to reduce the main theorem to a statement about graph invariants. By positivity of the canonical height, we find
	\be \label{Eq: Big Inequalities}
		\ba
		\inf_{D \in \Div^1\left(\alg{C}\right)} a'(D) &\geq \frac{(\omega_a, \omega_a)}{4(g-1)} 
				\hspace{5.15cm} \text{(Theorem~\ref{Thm: Admiss Pairing})}\\
			&\geq \frac{1}{2(2g+1)} \sum_{y \in Y(k)}  \Zphi{\agam_y}
				\hspace{2.6cm} \text{(Theorem~\ref{Thm: Gross-Schoen})} \\
			&= \frac{1}{2(2g+1)} \sum_{y \in Y(k)} \frac{1}{d} \sum_{\substack{y' \in Y'(k) \\ \alpha(y') = y}}
				e_{y'} \ \Zphi{\agam_y} \hspace{0.35cm} \Big(\text{as } 
					 \sum_{\alpha(y') = y} e_{y'} = d \Big) \\
			&=  \frac{1}{2d(2g+1)}\sum_{y' \in Y'(k)}  \ \Zphi{\agam_{y'}}
				\hspace{1.55cm} \text{(by \eqref{Eq: Ramification Claim})}.
		\ea
	\ee
Thus we have reduced the problem about algebraic curves to a problem about pm-graphs. In the remainder of the paper, we will prove the following result:
	
\begin{thm} \label{Thm: Graph Bounds}
	Let $c(2)=\frac{1}{27}$, $c(3)= \frac{2}{81}$, and $c(4) = \frac{1}{36}$. Then for any polarized metric graph 
	of genus $g = 2, 3$, or $4$, we have
		\benn
				\Zphi{\agam} \geq c(g) \ell_0\left(\agam\right)
					+\sum_{i \in (0, g/2]} \frac{2i(g-i)}{g}\ell_i\left(\agam\right).
		\eenn
\end{thm}
		
	If we assume Theorem~\ref{Thm: Graph Bounds} for the moment, then the relations~\eqref{Eq: Big Inequalities} and~\eqref{Eq: Singular lengths} show that for a curve $C / K$ of genus $2 \leq g \leq 4$, 
{\small
	\benn
		\ba
			\inf_{D \in \Div^1\left(\alg{C}\right)} a'(D) &\geq 
				\frac{1}{2d(2g+1)} \sum_{y' \in Y'(k)}  \left[c(g) \ell_0\left(\agam_{y'}\right)
					+\sum_{i \in (0, g/2]} \frac{2i(g-i)}{g}\ell_i\left(\agam_{y'}\right)\right] \\
				&=  \frac{1}{2d(2g+1)}\left[ c(g) \delta_0(X' / Y') 
					+ \sum_{i \in (0, g/2]}  \frac{2i(g-i)}{g} \delta_i(X'/Y') \right].
		\ea
	\eenn
}
\hspace{-4pt}Evidently this proves Theorem~\ref{Thm: Main Theorem}. Note that it gives a positive lower bound for $a'(D)$ whenever there exists a singular fiber of $f:X' \to Y'$, and so the Geometric Bogomolov Conjecture is true for curves of genus at most $4$ with bad reduction.


\section{Rapid Computation of $\Zphi{\agam}$}
\label{Sec: Rapid}

	Let $\agam = (\Gamma, q)$ be a pm-graph of genus $g$, and let us fix a model $G$. Enumerate the edges of $G$ as $e_1, \ldots, e_m$, and suppose these edges have lengths $\ell_1, \ldots, \ell_m$, respectively. By fixing the combinatorial type of $G$ as well as the function $q$, we may view $\Zphi{\agam}$ as a function of $\ell_1, \ldots, \ell_m$. It is our goal now to further illuminate the nature of this function. We will use our new description to prove ``continuity under edge contractions.'' This section is called \textit{rapid computation of $\Zphi{\agam}$} because we provide a reasonably efficient algorithm for its implementation in a computer algebra package. See Remarks~\ref{Rem: Algorithms1} and~\ref{Rem: Algorithms2}.
	
\begin{lem}
	Let $e$ be a segment of $\Gamma$ with respect to the model $G$. Choose an orientation on $G$ and use it to 
	give an isometry $y: [0, \ell(e)] \simarrow e$. For any vertex $p$ of $G$, we have
		\benn
			\int_e r(p,y) dy = \frac{\ell(e)^2}{6}F(e)
				+ \frac{\ell(e)}{2}\left( r(p, e^{-}) \stackrel{\phantom{+}}{+} r(p, e^+) \right).
		\eenn
\end{lem}

\begin{proof}
	By \eqref{Eqn: mu_0} and \eqref{Eqn: mu vs mu_0} in the proof of Proposition~\ref{Prop: Tau in terms of epsilon}, we know that 
	\benn
		\Delta_y\left( r(p,y) \right) = 2g\mu + \text{discrete masses}.
	\eenn
In particular, since we gave an explicit formula for $\mu$ in \S\ref{Sec: Admissible Stuff}, 
	\benn
		\ba
			& -\frac{d^2}{ds^2}r(p, y(s)) = 2\frac{F(e)}{\ell(e)} \\
			& \Rightarrow  r(p, y(s)) = - \frac{F(e)}{\ell(e)}s^2 + As + B, \quad 0 \leq s \leq \ell(e),
		\ea
	\eenn
for some constants $A$ and $B$ independent of $s$. Substituting $s=0$ shows $B = r(p, e^-)$, and then substituting $s = \ell(e)$ allows us to solve for $A$. Explicitly, we find
	\benn
		r(p, y(s)) = - \frac{F(e)}{\ell(e)}s^2 + \left( \frac{r(p, e^+) - r(p, e^-)}{\ell(e)} +F(e) \right)s + r(p, e^-).
	\eenn
Now it is a simple matter of calculus to compute the integral in the lemma and arrive at the desired expression for it.
\end{proof}

\begin{lem} \label{Lem: Computing epsilon}
	Let $\agam = (\Gamma, q)$ be a pm-graph of genus $g$. Fix a model $G$ of $\agam$ with vertex set
	$\{p_1, \ldots, p_n\}$ and edge set $\{e_1, \ldots, e_m\}$. Let $K(p_i) = v(p_i) - 2 + 2q(p_i)$ be the order of the
	canonical divisor at $p_i$. Then we have
	\benn
		\ba
			\Zepsilon{\agam} &= \frac{1}{g}\sum_{i,j}q(p_i) K(p_j) r(p_i, p_j) 
				+ \frac{g-1}{3g} \sum_{k} F(e_k)^2\ell_k \\
				& \qquad + \frac{1}{2g} \sum_i K(p_i) \sum_k F(e_k) \left(r(p_i, e_k^-) + r(p_i, e_k^+) \right).
		\ea
	\eenn
\end{lem}

\begin{proof}
	By definition, we have
		\benn
			\ba
				g \ \Zepsilon{\agam} &= g \int r(K, y) d\mu(y) \\
					&= g \sum_i K(p_i) \int r(p_i, y) d\mu(y) \\
					&= \sum_i K(p_i) \int r(p_i, y) \left( \sum_j q(p_j)\delta_{p_j}(y) 
						+ \sum_k \frac{F(e_k)}{\ell_k} dy|_{e_k} \right) \\
					&= \sum_{i, j} q(p_j) K(p_i) r(p_i, p_j) + \sum_i K(p_i) \sum_k \frac{F(e_k)}{\ell_k}
						\int_{e_k} r(p_i, y) dy.
			\ea
		\eenn
	Inserting the formula from the previous lemma and simplifying completes the proof.
\end{proof}

\begin{remark} \label{Rem: Algorithms1}
	Corollary~\ref{Cor: No tau} and Lemma~\ref{Lem: Computing epsilon} allow us to implement $\Zphi{\agam}$
	 in a computer algebra package. Indeed, all of the quantities involved are discrete in the sense that 
	they depend only on the finite quantity of data contained in the weighted graph $G$ and the function $q$. 
	Computing the effective resistance is easy since it is essentially an inverse of the combinatorial Laplacian 
	matrix. (Technically speaking, this is false since $Q$ is singular.) We will push these ideas further momentarily in 
	order to obtain more efficient algorithms for these computations.
\end{remark}

	We continue with the notation from the beginning of this section. Define
		\be \label{Eq: Def Eta}
			\eta(\ell_1, \ldots, \ell_m) = \sum_{T \subset G} \left( \prod_{e_k \not\subset T} \ell_k \right),
		\ee
where the sum is over all spanning trees $T$ of $G$. For example, $\eta(1, \dots, 1)$ is the number of spanning trees of $G$. As each spanning tree is the complement of $b_1 = \dim_\RR H^1(\Gamma, \RR)$ edges of $G$, we see that $\eta \in \ZZ[\ell_1, \ldots, \ell_m]$ is homogeneous of degree $b_1$. A useful alternative method for calculating $\eta$ is given by Kirchhoff's matrix-tree theorem. Let $Q$ be the combinatorial Laplacian matrix for $G$. If the vertices of $G$ are enumerated as $p_1, \ldots, p_n$, then $Q$ is an $n \times n$ matrix whose entries are given by 
	\benn
		Q_{ij} = \begin{cases} \sum_{e_k = \{p_i, *\}} \ell_k^{-1} & \text{if $i = j$} \\
			- \sum_{e_k = \{p_i, p_j\} } \ell_k^{-1} & \text{if $i \not= j$} \\
			0 & \text{otherwise.} \end{cases}
	\eenn
The first summation is over all non loop-edges $e_k$ containing the vertex $p_i$ and the second is over edges $e_k$ with vertices $p_i$ and $p_j$. 

	The kernel of $Q$ is generated by the vector $[1, 1, \ldots, 1]^T$, and a simple consequence is that all of the first cofactors of $Q$ are equal. (Use the easily verified fact $Q \cdot \operatorname{adj}(Q) = 0$.) Denote their common value by $\kappa^*(G)$. Then the matrix-tree theorem says
	\be \label{Eqn: Define Eta}
		\eta(\ell_1, \ldots, \ell_m) = \kappa^*(G) \prod_{k=1}^m \ell_k .
	\ee 
Compare \cite[\S II.3, Thm.~12]{Bollobas_Modern_Graph_Theory_1998}.

\begin{lem} \label{Lem: Modified Resistance}
	Let $\Gamma$ be a metric graph with first Betti number $b_1 = \dim_\RR H^1(\Gamma, \RR)$, and let $G$ be a model of $\Gamma$ with 
	vertex set $\{p_1, \ldots, p_n\}$ and edge set $\{e_1, \ldots, e_m\}$. For each pair
	of indices $1 \leq i,j \leq n$, define
		\benn
			R_{ij}(\ell_1, \ldots, \ell_m) =  r(p_i, p_j)\ \eta(\ell_1, \ldots, \ell_m).
		\eenn
	Then $R_{ij}$ is a homogeneous polynomial of degree $b_1+1$ with integer coefficients.	
\end{lem}

\begin{proof}
	For $i = j$, we see immediately that $R_{ij} = 0$. By symmetry, it suffices to prove the result for fixed $i$ and $j$ with $1 \leq i < j \leq n$. Set $y= p_i$ and $z = p_j$. By definition, $r(y,z) = j_z(y,y)$, where $f(x) = j_z(x,y)$ is the fundamental solution to the Laplace equation $\Delta_x f = \delta_y - \delta_z$ satisfying $f(z)=0$.  We may compute the value of $j_z (p_k, y)$ for any $k$ using the discrete Laplace equation
	\be \label{Eqn: Laplace}
		Q\mathbf{f} = \mathbf{e}_i - \mathbf{e}_j, \qquad f_j = 0,
	\ee
where $\mathbf{e}_k$ is the $k$th standard basis vector of $\RR^n$ and $\mathbf{f} = \sum_k f_k \mathbf{e}_k$. (Compare \cite[\S5--6]{Baker_Faber_2006}.) By definition, if $\mathbf{f}$ is the unique solution, then $j_z(p_k, y) = f_k$.

	The unique solution to the system of equations \eqref{Eqn: Laplace} gives rise to a solution of
	\be \label{Eqn: Modified Laplace}
		Q^{(j)}\mathbf{h} = \mathbf{e}_i,
	\ee
where $Q^{(j)}$ is the $(n-1) \times (n-1)$ matrix given by deleting the $j$th row and column from $Q$. The correspondence is given by 
	\[
	 	\mathbf{f} \mapsto \mathbf{h} = \sum_{k =1}^{j-1} f_k \mathbf{u}_k 
			+ \sum_{k = j+1}^n f_k \mathbf{u}_{k-1}.
	\]
Here $\mathbf{u}_k$ is the $k$th standard basis vector of $\RR^{n-1}$. Since $\det(Q^{(j)}) = \kappa^*(G) \not= 0$ by the matrix-tree theorem, the matrix $Q^{(j)}$ is invertible and the solution to \eqref{Eqn: Modified Laplace} is unique.

	Define $Q^{(j)}_i$ to be the matrix $Q^{(j)}$ with the $i$th column replaced by the $i$th standard basis vector $\mathbf{u}_i$. The above argument and Cramer's rule shows that if $\mathbf{h}$ is the solution to \eqref{Eqn: Modified Laplace}, 
	\benn
		\ba
			r(p_i, p_j) = h_i 
				&= \frac{\det(Q^{(j)}_i)}{\det(Q^{(j)})} 
				= \frac{\det(Q^{(j)}_i) \prod_{k=1}^m \ell_k}{\det(Q^{(j)}) \prod_{k=1}^m \ell_k} 
				= \frac{\det(Q^{(j)}_i) \prod_{k=1}^m \ell_k}{\eta(\ell_1, \ldots, \ell_m)} \\
				&\Longrightarrow R_{ij}(\ell_1, \ldots, \ell_m) = \det(Q^{(j)}_i) \prod_{k=1}^m \ell_k.
		\ea
	\eenn
To complete the proof of the lemma, we must show that $\det(Q^{(j)}_i) \prod_{k=1}^m \ell_k$ is a homogeneous polynomial of degree $b_1+1$ in the lengths $\ell_1, \ldots, \ell_m$ with integer coefficients.
	
	To see that $R_{ij}$ is a polynomial, we assume that $i=1$ and $j=2$, perhaps after relabeling the vertices. Let $G'$ be the graph given by fusing the vertices $p_1$ and $p_2$. If $p_{12}$ is the image of $p_1$ and $p_2$ in $G'$, then we may write $V(G') = \{p_{12}, p_3, \ldots, p_n\}$ and $E(G') = E(G)$. Let $Q'$ be the Laplacian matrix of $G'$. As the edges adjacent to $p_3, \ldots, p_n$ are unaffected by fusing $p_1$ and $p_2$, the lower right $(n-2) \times (n-2)$ submatrix of $Q'$ agrees with that of $Q_1^{(2)}$. Hence $\det(Q_1^{(2)}) = \kappa^*(G')$ in the notation preceding Lemma~\ref{Lem: Modified Resistance}. It follows that
	\be \label{Eqn: Poly Formula}
		R_{ij}(\ell_1, \ldots, \ell_m) = \det(Q_1^{(2)}) \prod_{k=1}^m \ell_k = \eta_{G'}(\ell_1, \ldots, \ell_m),
	\ee
where $\eta_{G'}$ is the polynomial associated to the graph $G'$ as in \eqref{Eqn: Define Eta}. The definition~\eqref{Eq: Def Eta} of $\eta_{G'}$ shows that it is a polynomial in $\ell_1, \ldots, \ell_m$ with integer coefficients. Moreover, the complement of any spanning tree in $G'$ consists of $b_1+1$ edges, and so $R_{ij}$ is homogeneous of degree $b_1+1$.
\end{proof}

\begin{lem} \label{Lem: Modified Edge Resistance}
	Let $\Gamma$ be a metric graph and $G$ be a model of $\Gamma$ with vertex set $\{p_1, \ldots, p_n\}$
	and edge set $\{e_1, \ldots, e_m\}$. For an edge $e_k = \{p_i, p_j\}$ of $G$, set
	$R_k=  R_{ij}$ to be the polynomial defined in the previous lemma. Then $\ell_k$ divides $R_k$ as polynomials
	in $\ZZ[\ell_1, \ldots, \ell_m]$ and 
		\benn
			\frac{R_k(\ell_1, \ldots, \ell_m)}{\ell_k} = \eta(\ell_1, \ldots, \ell_m)|_{\ell_k = 0}.
		\eenn
\end{lem}

\begin{proof}
	The proof is similar to that of the previous lemma, except that we will contract the edge $e_k$, rather than fusing its endpoints. After reordering the edges and vertices if necessary, we may assume that the edge of interest is $e_1 = \{p_1, p_2\}$. Using the notation and strategy of the previous proof, we find 
	\benn
		\frac{R_1(\ell_1, \ldots, \ell_m)}{\ell_1} = \det(Q_1^{(2)})\prod_{k=2}^m \ell_k. 
	\eenn
	
	Define a new graph $G/e_1$ given by contracting the edge $e_1$. Let $p_{12}$ be the image of the vertices $p_1$ and $p_2$ in $G/e_1$. Then we may make the identifications $E(G/e_1) = E(G) \smallsetminus \{e_1\}$ and $V(G/ e_1) = \{p_{12}, p_3, \ldots, p_n\}$. Let $Q(G/ e_1)$ be the Laplacian matrix of $G/e_1$. It is not difficult to see that the lower right $(n-2) \times (n-2)$ submatrix of $Q(G/e_1)$ agrees with that of $Q_1^{(2)}$, and so $\det(Q_1^{(2)}) = \kappa^*(G/e_1)$ in the notation preceding Lemma~\ref{Lem: Modified Resistance}. In particular, 
	\benn
		\frac{R_1(\ell_1, \ldots, \ell_m)}{\ell_1} = \eta_{G/e_1}(\ell_2, \ldots, \ell_m),
	\eenn
where $\eta_{G/e_1}$ is the polynomial associated to the graph $G / e_1$ as in \eqref{Eqn: Define Eta}.

	To complete the proof, we must show that $\eta_{G/e_1}(\ell_2, \ldots, \ell_m) = \eta(0, \ell_2, \ldots, \ell_m)$. But this follows from its definition and the bijective correspondence between spanning trees of $G / e_1$ and spanning trees of $G$ containing the edge $e_1$.
\end{proof}

\begin{prop} \label{Prop: Special Forms}
	Let $\agam = (\Gamma, q)$ be a pm-graph of genus $g \geq 1$ and first Betti number $b_1$. Fix a model $G$ 
	of  $\agam$. With the notation above, there exists a homogeneous polynomial 
	$\omega_1 \in \QQ[\ell_1, \ldots, \ell_m]$ of degree $2b_1+1$, depending only on the combinatorial type of $G$
	and the function $q$, such that
		\benn
			\Zphi{\agam} = \frac{g-1}{6g} \ell(\Gamma) 
				- \frac{\omega_1(\ell_1, \ldots, \ell_m)}{\eta(\ell_1, \ldots, \ell_m)^2}.
		\eenn
\end{prop}

\begin{proof}
	Write $\eta = \eta(\ell_1, \ldots, \ell_m)$ for simplicity. Let us begin by defining $\omega_1$ as the following
 function of $\ell_1, \ldots, \ell_m$:
		\benn
			\omega_1(\ell_1, \ldots, \ell_m) := \frac{g-1}{6g}\ell(\Gamma) \ \eta^2 - \Zphi{\agam}\eta^2.
		\eenn
We must prove that $\omega_1$ is a homogeneous polynomial of degree $2b_1+1$ with rational coefficients. 

	As $\ell(\Gamma) = \ell_1 + \cdots +\ell_m$, it is clear that $\ell(\Gamma) \ \eta^2$ is homogeneous of the correct degree, so we are reduced to showing that $\Zphi{\agam} \eta^2$ has the same property. By Corollary~\ref{Cor: No tau} it suffices to verify that $\Zepsilon{\agam} \eta^2$ and $r(K,K) \  \eta^2$ are homogeneous 
degree-$(2b_1+1)$ polynomials with rational coefficients. The latter is true by 
Lemma~\ref{Lem: Modified Resistance}.
	
	By Lemma~\ref{Lem: Computing epsilon} and the fact that $F(e_k) = 1 - r(e_k) / \ell_k$, we have
	\benn
		\ba
			\Zepsilon{\agam} \eta^2 &= \frac{\eta}{g}\sum_{i,j} q(p_i) K(p_j) R_{ij} +
				\frac{g-1}{3g} \sum_k \left( \eta - \frac{R_k}{\ell_k}\right)^2 \ell_k \\
				& \qquad + \frac{1}{2g} \sum_i K(p_i) \sum_k \left( \eta - \frac{R_k}{\ell_k}\right)
					\Big(\eta \cdot r(p_i, e_k^-) + \eta \cdot r(p_i, e_k^+) \Big).
		\ea
	\eenn
Referring once more to Lemmas~\ref{Lem: Modified Resistance} and~\ref{Lem: Modified Edge Resistance}, we see that this last expression is indeed a polynomial of the correct degree in the lengths $\ell_1, \ldots, \ell_m$ with rational coefficients. 
\end{proof}

\begin{remark}
	Having computed a number of examples, it seems to be the case that the polynomial $\omega_1$ 
	defined in the previous proposition is divisible by the polynomial $\eta$. While
	we cannot prove this in general, a strategy to prove it in the special case of hyperelliptic graphs is suggested by
	\cite[\S3.3]{Yamaki_Bogomolov_Hyperelliptic_2008}. 
\end{remark}

\begin{remark} \label{Rem: Algorithms2}
	The results of this section and their proofs indicate the algorithms we have implemented in \textit{Mathematica}.
	In all of these algorithms, the input is a weighted graph. If for example one wanted to 
	calculate $\Zphi{\agam}$ for a metric graph, the best way to do it is to pick a model of $\agam$ with very few 
	vertices, and then apply these implementations to the model. 
	\bi
		\item[1.] The algorithm used to calculate effective resistance is given by the beginning of the proof of 
			Lemma~\ref{Lem: Modified Resistance}. 
			
		\item[2.] To calculate $\Zepsilon{\agam}$, we use the formula given by the proof of
			Proposition~\ref{Prop: Special Forms}. 
			
		\item[3.] To calculate $\Zphi{\agam}$, we use the proof of 
		Proposition~\ref{Prop: Special Forms} to compute the polynomial $\omega_1$ and
		then apply the formula in the statement of the proposition. See the next remark for an explicit formula
		for $\omega_1$.
	\ei
\end{remark}

\begin{remark} \label{Rem: Explicit Formula}
	For easy reference and later use, we now give an explicit formula for the polynomial $\omega_1$
	associated to a metric graph $\agam = (\Gamma, q)$ of genus $g$. This formula is easily derived from 
	the arguments in this section. For simplicity, we use slightly different notation here; let 
	$R(x,y) = \eta \cdot r(x,y)$. Also, we write $e = \{e^-, e^+\}$ for some choice of orientation on each edge $e$. 
	Then
		\benn
			\ba
				\omega_1 &= \frac{5g-2}{12g}\left[ \ell(\Gamma) \ \eta^2
					 	- \sum_k \left(\eta - \frac{R(e_k^-, e_k^+)}{\ell_k} \right)^2 \ell_k \right]\\
					 & \qquad + \frac{\eta}{8g(g-1)}\sum_{i,j} K(p_i) R(p_i, p_j) 
					 	\Big( 3gK(p_j) -2(5g-2) q(p_j) \Big) \\
					 & \qquad - \frac{5g-2}{8g(g-1)} \sum_i K(p_i) \sum_k 
						\left(\eta - \frac{R(e_k^-, e_k^+)}{\ell_k}\right)
						\Big(R(p_i, e_k^-) + R(p_i, e_k^+) \Big).
			\ea
		\eenn
	By symmetry, this formula is independent of the choice of edge orientations.

\end{remark}

	
\section{The Proof of Theorem~\ref{Thm: Graph Bounds}}
\label{Sec: Reductions}

	Our goal for this section is to give the proof of Theorem~\ref{Thm: Graph Bounds}, which we restate for the reader's convenience. 

\begin{thm*}
	Let $c(2)=\frac{1}{27}$, $c(3)= \frac{2}{81}$, and $c(4) = \frac{1}{36}$. Then for any polarized metric graph 
	of genus $g = 2, 3$, or $4$, we have
		\benn
				\Zphi{\agam} \geq c(g) \ell_0\left(\agam\right)
					+\sum_{i \in (0, g/2]} \frac{2i(g-i)}{g}\ell_i\left(\agam\right).
		\eenn
\end{thm*}

\begin{remark}
	The constant $c(2)$ is sharp, but $c(3)$ and $c(4)$ are not optimal due to the use of a wasteful estimate in 	
	Proposition~\ref{Prop: Verify} below. For example, further computational evidence suggests that one could 
	instead use $c(3) = 17 / 288$, but this has not been rigorously verified.
\end{remark}

	The inequality in the theorem is a special case of the following more general conjecture:
	
\begin{conjecture}[{\cite[Conj.~4.1.1]{Zhang_Gross-Schoen_Cycles_2008}}] 	\label{Conj: Zhang}
	Let $g \geq 2$ be an integer. There exists a positive constant $c(g)$ such that for any pm-graph $\agam$ of 
	genus $g$, 
		\benn
				\Zphi{\agam} \geq c(g) \ell_0\left(\agam\right)
					+\sum_{i \in (0, g/2]} \frac{2i(g-i)}{g}\ell_i\left(\agam\right).
		\eenn
\end{conjecture}

\begin{remark}
	Based on empirical evidence, we assert further that the conjecture should hold for polarized metric graphs of 
	arbitrary genus $g \geq 2$ with
		\benn
			c(g) = \frac{g-1}{27g}.
		\eenn
	This is discussed at the end of Section~\ref{Sec: Strategy}.
\end{remark}

Our plan is to show that, in order to prove the conjecture, it suffices to prove it for a finite simple class of graphs of genus $g$. (See Proposition~\ref{Reduce conjecture}.) This will require a number of reduction steps. In the first part, we give a number of results on the continuity of certain metric graph invariants with respect to edge contraction. In the second part we show that, in order to give a lower bound for $\Zphi{\agam}$ for all pm-graphs of a given genus $g$, it suffices to give a lower bound when $\agam = (\Gamma, q)$ is an irreducible cubic pm-graph of genus $g$ with $q$ identically zero. The basic idea is to reduce first to the case where $\Gamma$ has no points of type $i > 0$ by an additivity result on its components. Then we reduce to the case $q \equiv 0$ by adjoining circles along the support of $q$. Finally, we observe that an arbitrary metric graph with these properties can be obtained from an irreducible cubic one by letting some of the edge lengths tend to zero. This last step uses the continuity results in the first part. A similar strategy was used in \cite{Cinkir_Thesis_2007, Yamaki_Bogomolov_Hyperelliptic_2008, Yamaki_Bogomolov_Genus_3_2002} to reduce their respective questions to the case of cubic graphs.  In the third part we give the strategy used to complete the proof of Theorem~\ref{Thm: Graph Bounds} for graphs of a fixed genus. It requires a bit of difficult computation which we can accomplish when the genus is at most $4$. In the fourth part we summarize these calculations, having performed them in \textit{Mathematica}. The \textit{Mathematica} code and the actual calculation notebooks close out the article. 
	
\subsection{Continuity with Respect to Edge Contraction}
\label{Sec: Continuity}

	In this section we restrict our attention to pm-graphs $\agam = (\Gamma, q)$ with $q \equiv 0$ and genus $g \geq 2$. For simplicity, we will write $\agam = (\Gamma, 0)$. The results hold for an arbitrary pm-graph with only slight modification, but we have no need for the general case. 
	
	Let $\agam_1 = (\Gamma_1, 0)$ be a pm-graph, and let $G$ be a model of $\agam_1$. Fix an edge $e_1$ of $G$ \textit{that is not a loop edge}, and let us identify it with the corresponding segment of $\Gamma_1$. Define a family of pm-graphs $\agam_t = (\Gamma_t, 0)$ as follows. For each $t > 0$, the combinatorial structure of $\Gamma_t$ is given by $G$, but the length of the segment $e_1$ is $t$. For $t=0$, let $\Gamma_0$ be the metric graph given by contracting the segment $e_1$ to a point; the weighted graph $G / e_1$ serves as a model of $\Gamma_0$. Let $e_1 = \{p_1, p_2\}$, and let $p_{12}$ be the image of the segment $e_1$ under the contraction $\Gamma_1 \to \Gamma_0$. As $e_1$ is not a loop edge, the family $\Gamma_t$ defines a homotopy from $\Gamma_1$ to $\Gamma_0$. In particular, $g(\agam_t)$ is constant for all $t$. 
		
\begin{prop} \label{Prop: Continuity}
	With notation as in the previous paragraph,
		\benn
			\Zphi{\agam_0} = \lim_{t \to 0} \Zphi{\agam_t}.
		\eenn
\end{prop}

	We will spend the rest of the section proving this result. Our strategy will be to use Proposition~\ref{Prop: Special Forms} and prove each of the quantities that appears there behaves well as $t$ tends to zero. Each of these will be a separate lemma. 
	
\begin{lem}
		\benn
			\lim_{t \to 0} \ell(\Gamma_t) =  \ell(\Gamma_0).
		\eenn
\end{lem}

\begin{proof}
	The definitions immediately give $\ell(\Gamma_t) = t + \ell(\Gamma_0)$.
\end{proof}

	Suppose the vertex and edge sets of $G$ are given by $\{p_1, \ldots, p_n\}$ and $\{e_1, \ldots, e_m\}$, respectively. Let $\ell_1, \ell_2, \ldots, \ell_m$ be the segment lengths of $\Gamma_t$, where $\ell_1 = t$. We also have $E(G/e_1) = E(G) \smallsetminus \{e_1\}$ and $V(G/e_1) = \{p_{12}, p_3, \ldots, p_n\}$. Define
	\benn
		\eta(\Gamma_t) = \begin{cases} \sum_{T \subset G} \prod_{e_k \not\subset T} \ell_k & \text{if $t>0$,} \\ 
			\sum_{T \subset G/e_1} \prod_{e_k \not\subset T} \ell_k & \text{if $t=0$,} \end{cases}
	\eenn
where the summations are over spanning trees $T$. Note that $\eta(\Gamma_t) \in \ZZ[t, \ell_2, \ldots, \ell_m]$. 

\begin{lem}
	\benn
		\lim_{t \to 0} \eta(\Gamma_t) = \eta(\Gamma_0)
	\eenn
\end{lem}

\begin{proof}
	Since the spanning trees of $G/e_1$ are in bijection with the spanning trees of $G$ containing the edge $e_1$, we see that as $t \to 0$, 
	\benn
		\ba		
			\eta(\Gamma_t) &= \sum_{e_1 \subset T \subset G} 
				\ \prod_{e_k \not\subset T} \ell_k + t \sum_{\substack{T \subset G \\ e_1 \not\subset T}} 
				\ \prod_{\substack{e_k \not\subset T \\ k \not= 1}} \ell_k \\
			&= \sum_{T \subset G/e_1} \prod_{e_k \not\subset T} \ell_k + o(1) \\
			&= \eta(\Gamma_0) + o(1).
		\ea
	\eenn
\end{proof}

	Let $r_t(x,y)$ denote the effective resistance between points $x,y \in \Gamma_t$. Define 
		\[
			R_t(x,y) = \eta(\Gamma_t)\cdot r_t(x,y). 
		\]
We have already seen in Lemma~\ref{Lem: Modified Resistance} that $R_t(p, p')$ is a polynomial in the lengths $t, \ell_2, \ldots, \ell_m$ for any vertices $p, p' \in G$ if $t > 0$ (respectively $p, p' \in G/e_1$ if $t=0$). 
	
\begin{lem} \label{Lem: Edge Contraction}
	\benn
		\lim_{t \to 0} R_t(p_i, p_j) = \begin{cases} 0 & \text{if $\{i, j\} \subset \{1, 2\}$} \\
			R_0(p_{12}, p_j) & \text{if $i \in \{1, 2\}$ and $j \geq 3$} \\
			R_0(p_i, p_{12}) & \text{if $j \in \{1, 2\}$ and $i \geq 3$} \\
			R_0(p_i, p_j) & \text{if $i, j \geq 3$}.
		\end{cases}
	\eenn
\end{lem}

\begin{proof}
	This result is intuitively obvious from the vantage point of circuit theory since the effective resistance behaves continuously under edge contractions. We will give a rigorous proof using the ideas from Lemma~\ref{Lem: Modified Resistance}.
	
	If $p_i = p_j$, then $R_t(p_i, p_j) = 0$ for all $t > 0$. If $i,j \geq 3$, then $R_0(p_i, p_j) = 0$ as well. Since we are in either the first or fourth cases of the lemma, we have proved the statement.

	For the remainder of the proof we assume $p_i \not= p_j$. By \eqref{Eqn: Poly Formula} in the proof of Lemma~\ref{Lem: Modified Resistance}, we find
	\[
		R_t(p_i, p_j) = \eta_{G'}(t, \ell_2, \ldots, \ell_m),	
	\]
where $G'$ is the graph given by fusing the vertices $p_i$ and $p_j$. Thus
	\be \label{Eqn: Limit Formula}
		\ba
			\lim_{t \to 0} R_t(p_i, p_j) &= \eta_{G'}(0, \ell_2,\ldots, \ell_m) \\
				&= \sum_{T \subset G'} \ \prod_{e_k \not\subset T} \ell_k\  \Big|_{\ell_1 = 0} \\
				&= \sum_{e_1 \subset T \subset G'} \ \prod_{e_k \not\subset T} \ell_k.
		\ea
	\ee
The first summation is over all spanning trees of $G'$ while the second is only over the spanning trees containing $e_1$. 

	Let us now assume that $\{i,j\} = \{1,2\}$. Then the edge $e_1$ becomes a loop edge in $G'$, and so no spanning tree contains $e_1$. The last expression in \eqref{Eqn: Limit Formula} must be zero, which is the desired result.
	
	Finally, we assume we are in one of the remaining three cases of the lemma. Then $e_1$ is not a loop edge in $G'$. Let $G' / e_1$ be the graph derived from $G'$ by collapsing the edge $e_1$. It is naturally identified with the graph $(G / e_1)'$ given by first collapsing $e_1$ and then fusing the images of $p_i$ and $p_j$.  The spanning trees of $G'$ containing $e_1$ are in bijection with the spanning trees of $G' / e_1 = (G / e_1)'$. From \eqref{Eqn: Limit Formula} we see that
	\benn
		\ba
		\lim_{t \to 0} R_t(p_i, p_j) &= \sum_{e_1 \subset T \subset G'} \ \prod_{e_k \not\subset T} \ell_k \\
			&= \sum_{T \subset (G / e_1)'} \ \prod_{e_k \not\subset T} \ell_k \\
			&= \eta_{(G / e_1)'} \left(\ell_2, \ldots, \ell_m\right) 
				\hspace{0.7cm} \text{(definition of $\eta_{(G/e_1)'}$)}\\
			&= R_0(\bar{p}_i, \bar{p}_j),
				 \hspace{2.25cm} \text{(by \eqref{Eqn: Poly Formula} again),}
		\ea
	\eenn
where $\bar{p}_i$ and $\bar{p}_j$ are the images of $p_i$ and $p_j$, respectively, in the limit graph $\Gamma_0$. This completes the proof of the lemma.
\end{proof}

	Let $\omega_1(\agam_t) \in \QQ[t, \ell_2, \ldots, \ell_m]$ be the polynomial defined in Proposition~\ref{Prop: Special Forms} with respect to the model $G$ when $t > 0$ (respectively $G/e_1$ when $t=0$).

\begin{lem} 
	In the setting above,
	\benn
		\lim_{t \to 0} \omega_1(\agam_t) = \omega_1(\agam_0).
	\eenn
\end{lem}

\begin{proof}
	Remark~\ref{Rem: Explicit Formula} gives a formula for $\omega_1$. Let $K_t$ be the canonical divisor for the graph $\agam_t$.  As Lemma~\ref{Lem: Modified Edge Resistance} and its proof imply $\eta(\Gamma_0) = R_t(e_1^-, e_1^+) / t$, we have
{\small 
	\be \label{Eqn: Four expressions}
		\ba
			\omega_1(\agam_t) &= \frac{5g-2}{12g}\left[ \ell(\Gamma_t) \ \eta(\Gamma_t)^2
						 - t \Big( \eta(\Gamma_t) -  \eta(\Gamma_0) \Big)^2 
						- \sum_{k \geq 2} \left(\eta(\Gamma_t) 
						- \frac{R_t(e_k^-, e_k^+)}{\ell_k} \right)^2 \ell_k \right]\\
					 & \qquad + \frac{3\eta(\Gamma_t)}{8(g-1)}\sum_{i,j} K_t(p_i) K_t(p_j) R_t(p_i, p_j) \\
					 & \qquad - \frac{5g-2}{8g(g-1)} 
					 	\Big(\eta(\Gamma_t) - \eta(\Gamma_0) \Big) \sum_i K_t(p_i)
						\Big(R_t(p_i, e_1^-) + R_t(p_i, e_1^+) \Big)\\
					 & \qquad - \frac{5g-2}{8g(g-1)} \sum_i K_t(p_i) \sum_{k \geq 2} 
						\left(\eta(\Gamma_t) - \frac{R_t(e_k^-, e_k^+)}{\ell_k}\right)
						\Big(R_t(p_i, e_k^-) + R_t(p_i, e_k^+) \Big).
		\ea
	\ee
}
We will treat each of these four expressions separately using the previous several lemmas. 

	For the first expression of \eqref{Eqn: Four expressions}, we have
	\benn 	\label{Eqn: First expression}
		\ba
			\frac{5g-2}{12g}&\left[ \ell(\Gamma_t) \ \eta(\Gamma_t)^2
						 - t \Big( \eta(\Gamma_t) -  \eta(\Gamma_0) \Big)^2 
						- \sum_{k \geq 2} \left(\eta(\Gamma_t) 
						- \frac{R_t(e_k^-, e_k^+)}{\ell_k} \right)^2 \ell_k \right]\\ & = 
			\frac{5g-2}{12g}\left[ \ell(\Gamma_0) \eta(\Gamma_0)^2 - \sum_{k \geq 2} \left(\eta(\Gamma_0) - 
				\frac{R_0(e_k^-, e_k^+)}{\ell_k}\right)^2\ell_k \right]+ o(1),
		\ea
	\eenn
as $t \to 0$. Thus we obtain the correct first expression for $\omega_1(\Gamma_0)$. Moving to the second expression, we note that $K_t(p_i) = K_0(p_i)$ for $i \geq 3$ and that 
	\[
		K_t(p_1) + K_t(p_2) = v(p_1) +v(p_2) - 4 = v(p_{12}) - 2 = K_0(p_{12}).
	\]
Ignoring the leading factor on the second expression and using Lemma~\ref{Lem: Edge Contraction}, we see
{\small
	\benn \label{Eqn: Second expression}
		\ba
			\sum_{i,j} &K_t(p_i) K_t(p_j)  R_t(p_i, p_j)  = 
					2K_t(p_1)K_t(p_2)R_t(p_1, p_2) + \sum_{i,j \geq 3} K_t(p_i)K_t(p_j) R_t(p_i, p_j) \\
				& \qquad + 2\sum_{i \geq 3} K_t(p_1)K_t(p_i)R_t(p_1, p_i) 
					+ 2\sum_{i \geq 3} K_t(p_2)K_t(p_i)R_t(p_2, p_i) \\
			 	&= \sum_{i,j \geq 3} K_0(p_i)K_0(p_j) R_0(p_i, p_j)  
					+ 2\sum_{i \geq 3}\Big[K_t(p_1) + K_t(p_2)\Big] K_0(p_i)R_0(p_{12}, p_i) 
					+o(1) \\
				&= \sum_{i,j \geq 3} K_0(p_i)K_0(p_j) R_0(p_i, p_j)  
					+ 2\sum_{i \geq 3} K_0(p_{12})K_0(p_i)R_0(p_{12}, p_i) + o(1).
		\ea
	\eenn
}
Thus we find the second expression tends to the correct limit as $t \to 0$. 

	The third expression in \eqref{Eqn: Four expressions} vanishes as $t$ tends to zero because of the factor $\eta(\Gamma_t) - \eta(\Gamma_0)$. The fourth expression tends to the correct limit using the same idea as in the second expression (i.e., group the terms corresponding to $i=1, 2$ separately).
\end{proof}

\begin{proof}[Proof of Proposition~\ref{Prop: Continuity}]

	By Proposition~\ref{Prop: Special Forms} and the previous lemmas, we have
	\benn
		\ba
			\Zphi{\agam_t} &= \frac{g-1}{6g}\ell(\Gamma_t) - \frac{\omega_1(\agam_t)}{\eta(\Gamma_t)^2}
				= \frac{g-1}{6g}\ell(\Gamma_0) - \frac{\omega_1(\agam_0)}{\eta(\Gamma_0)^2} + o(1) \\
				&= \Zphi{\agam_0} + o(1)
		\ea
	\eenn
as $t \to 0$. Here we have used the fact that $\eta(\Gamma_0) \not= 0$, which is evident by its definition and the fact that none of the other edge lengths is zero.
\end{proof}

	
\subsection{Reduction to Irreducible Cubic PM-Graphs}

	In this section we show the following proposition:
	
\begin{prop} \label{Reduce conjecture}
	Fix $g \geq 2$. Suppose there is a positive constant $c(g)$ such that for all irreducible cubic pm-graphs 
	$\agam = (\Gamma, 0)$ of genus $g$,
		\benn
			\Zphi{\agam} \geq c(g)\ell(\Gamma).
		\eenn 
	Then for an arbitrary pm-graph $\agam$, we have
		\benn
				\Zphi{\agam} \geq c(g) \ell_0\left(\agam\right)
					+\sum_{i \in (0, g/2]} \frac{2i(g-i)}{g}\ell_i\left(\agam\right).
		\eenn
	In particular, it suffices to verify Conjecture~\ref{Conj: Zhang} only for cubic irreducible 
	pm-graphs $\agam = (\Gamma, 0)$. 
\end{prop}

Our first step is to summarize some useful observations of Zhang \cite[\S4.3]{Zhang_Gross-Schoen_Cycles_2008}. Let $\agam = (\Gamma, q)$ be an arbitrary pm-graph. Consider a finite collection of (closed, connected) metric subgraphs $\Gamma_1, \ldots, \Gamma_s$ such that $\Gamma$ is the successive pointed sum of the $\Gamma_i$. For example, if $\Gamma$ can be disconnected by removing a point $p$, then we may write $\Gamma = \Gamma_1 \cup \Gamma_2$ for some subgraphs $\Gamma_i$ such that $\Gamma_1 \cap \Gamma_2 = \{p\}$. In general, we may assign a pm-graph structure to each of the subgraphs $\Gamma_i$ as follows. Let $\pi_i: \Gamma \to \Gamma_i$ be the retraction map sending each point in $\Gamma$ to the closest point in $\Gamma_i$. Define a new function $q_i: \Gamma_i \to \ZZ$ by 
	\benn
		q_i(y) 
			= \dim_{\RR} H^1(\pi_i^{-1}(y), \RR) + \sum_{x \in \pi_i^{-1}(y)}q(x).
	\eenn
With these definitions, we have

\begin{lem}[{\cite[Thm.~4.3.2]{Zhang_Gross-Schoen_Cycles_2008}}] \label{Lem: Additivity}
	Each pair $\agam_i = (\Gamma_i, q_i)$ is a pm-graph and $g(\agam) = g(\agam_i)$. Moreover,
		\benn
			\Zphi{\agam} = \sum_i \Zphi{\agam_i}.
		\eenn
\end{lem}
	
\begin{lem}[{\cite[Lem.~4.3.1, Prop.~4.4.1]{Zhang_Gross-Schoen_Cycles_2008}}] \label{Lem: Decomp}
	Let $\agam = (\Gamma, q)$ be a pm-graph of genus $g$. Then $\Gamma$ can be written as a successive pointed 
	sum of finitely many subgraphs $\Gamma_i$ and $I_j$ such that each $\Gamma_i$ is a maximal irreducible 
	subgraph of $\Gamma$ and each $I_j$ is isometric to a closed interval for which $q$ is identically zero on its 
	interior. Moreover,
		\benn
			\Zphi{\agam} = \sum_i \Zphi{\agam_i} + \sum_{i \in (0, g/2]} \frac{2i(g-i)}{g}\ell_i\left(\agam\right).			\eenn
\end{lem}

We extend our terminology slightly from \S\ref{Sec: Main Definitions}. Let $\agam = (\Gamma, q)$ be a pm-graph and $p \in \Gamma$ any point. Then $p$ is of \textbf{positive type} if it is a point of type $i$ for some $i > 0$. Equivalently, $p$ is of positive type if $p$ has valence $2$, $q(p) = 0$, and $\Gamma \smallsetminus \{p\}$ is disconnected. Note that an irreducible metric graph has no points of positive type, but the converse is false. For example, the pointed sum of two circles can be disconnected by deleting their common point, but the type of this point is not defined.

\begin{lem} \label{Lem: All type 0}
	Fix $g \geq 2$. Suppose there exists a positive constant $c(g)$ such that for every pm-graph 
	$\agam = (\Gamma, q)$ of genus $g$ with no points of positive type, 
		\benn
			\Zphi{\agam} \geq c(g) \ell(\Gamma).
		\eenn
	Then for an arbitrary pm-graph of genus $g$, we have
		\benn
			\Zphi{\agam} \geq c(g)\ell_0\left(\agam\right) + \sum_{i \in (0, g/2]} \frac{2i(g-i)}{g}\ell_i\left(\agam\right).	
		\eenn
\end{lem}
			
\begin{proof}
	Let $\agam = (\Gamma, q)$ be an arbitrary pm-graph of genus $g$. By Lemma~\ref{Lem: Decomp} we have		\benn
		\Zphi{\agam} = \sum_j \Zphi{\agam_j} + \sum_{i \in (0, g/2]} \frac{2i(g-i)}{g}\ell_i\left(\agam\right),
	\eenn
where $\Gamma_j$ are the maximal irreducible subgraphs of $\Gamma$. By irreducibility, every point of $\Gamma_j$ is of type~$0$ or else the type is undefined. We may apply the hypothesis of the lemma to conclude that $\Zphi{\agam_j} \geq c(g) \ell(\Gamma_j)$. Evidently $\ell_0\left(\agam\right) = \sum_j \ell(\Gamma_j)$, which completes the proof.
\end{proof}

	Next we reduce to the case where $q$ is identically zero. We first provide an explicit formula in the special case when $\agam$ is a circle and $q$ is supported at a single point.
	
\begin{lem} \label{Lem: Constants for Circles}
	Let $C$ be a circle of length $\ell(C)$, let $p \in C$ be a point, and let $g \geq 1$ be an integer. 
	Set $\alg{C} = (C, q)$ to be the $pm$-graph with $q(p) = g-1$ and $q(x) = 0$ for $x \not=p$. Then 
		\benn
			\Zphi{\alg{C}}  =  \frac{g-1}{6g}\ell(C).
		\eenn
\end{lem}
	
\begin{proof}
	This computation is sufficiently simple to be done by hand and we leave it as an exercise for the reader. Alternatively, one could look at \cite[Prop.~4.4.3]{Zhang_Gross-Schoen_Cycles_2008}, of which the present lemma is a special case.
\end{proof}

\begin{lem} \label{Lem: q equals 0}
	Fix $g \geq 2$. Suppose there exists a positive constant $c(g)$ such that for every pm-graph 
	$\agam = (\Gamma, 0)$ of genus $g$ with no points of positive type, 
		\be \label{Eqn: q reduction}
			\Zphi{\agam} \geq c(g) \ell(\Gamma).
		\ee
	Then the above inequality holds for an arbitrary pm-graph of genus $g$ with no points of positive type. 
\end{lem}	

\begin{proof}
	We proceed by induction on the quantity $\sum_{x \in \Gamma} q(x)$. If $\sum_{x \in \Gamma} q(x)=0$, then $q \equiv 0$ and \eqref{Eqn: q reduction} holds by hypothesis. Let us suppose that we have proved \eqref{Eqn: q reduction} holds for all pm-graphs $\agam=(\Gamma, q)$ whenever $\Gamma$ has no points of positive type and $\sum_{x \in \Gamma} q(x) \leq N$. Choose $\agam = (\Gamma, q)$ to be a pm-graph with $N+1 =\sum_{x \in \Gamma} q(x)$ and no points of positive type. Then there is a point $p \in \Gamma$ with $q(p) >0$. For each $t > 0$, define a new pm-graph $\agam_t = (\Gamma_t, q_t)$ as follows:
	\bi
		\item The metric graph $\Gamma_t$ is obtained from $\Gamma$ by gluing a circle of length
			$t$ to the point $p$. Label the circle $C_t$.
								
		\item Set 
				\[
						q_t(x) =  \begin{cases}
					q(x) & \text{if $x \in \Gamma \smallsetminus \{p\}$} \\
					q(p)- 1 & \text{if $x=p$.} \\
					0 &  \text{if $x \in C_t \smallsetminus \{p\}$} \\
					\end{cases}
				\]
	\ei
Then $\agam_t$ is a pm-graph of genus $g$ and $\sum_{x \in \Gamma_t} q(x) = N$. Moreover, every point of $C_t \smallsetminus\{p\}$ has type~0, and so $\Gamma_t$ has no points of positive type. By the induction hypothesis, the inequality~\eqref{Eqn: q reduction} holds for $\Gamma_t$.

	Define $\alg{C}_t = (C_t, q')$ by setting $q'(p) = g-1$ and $q'(x) = 0$ otherwise. Then the induced pm-graph structures on the metric subgraphs $\Gamma$ and $C_t$ of $\Gamma_t$ are given by $\agam$ and $\alg{C}_t$, respectively. By the additivity described in Lemma~\ref{Lem: Additivity}, we have
	\benn
		\ba
			 \Zphi{\agam_t} &= \Zphi{\agam}+ \Zphi{\alg{C}_t} \\
			\Rightarrow \Zphi{\agam} &= \Zphi{\agam_t} - \Zphi{\alg{C}_t} \\
				&\geq c(g)\ell(\Gamma_t) -  \frac{g-1}{6g}t  
					\quad \text{by Lemma~\ref{Lem: Constants for Circles}} \\
				& = c(g) \ell(\Gamma) +  \left( c(g) -\frac{g-1}{6g} \right)t.
		\ea
	\eenn
Let $t \to 0$ to obtain the desired bound for $\Zphi{\agam}$. This completes the induction step.
\end{proof}
		
\begin{lem} \label{Lem: Reduction Lemma}
	Fix $g \geq 2$. Suppose there exists a positive constant $c(g)$ such that for every irreducible cubic 
	pm-graph $\agam = (\Gamma, 0)$ of genus $g$,   
		\be	\label{Eq: Phi Inequality}
			\Zphi{\agam} \geq c(g) \ell(\Gamma).
		\ee
	Then the above inequality holds for an arbitrary pm-graph $\agam = (\Gamma, 0)$ of genus $g$ with no 
	points of positive type.
\end{lem}

\begin{proof}
	Let $\agam = (\Gamma, 0)$ be a pm-graph of genus $g$ with no points of positive type. It follows that if $\Gamma \smallsetminus \{p\}$ is disconnected for some $p \in \Gamma$, then the valence of $p$ is at least $3$. If moreover $\Gamma$ is cubic, then it must be irreducible. Indeed, if $\Gamma \smallsetminus \{p\}$ is disconnected, then $p$ is a point of valence~$3$. There is a component of $\Gamma \smallsetminus \{p\}$ that contributes a single tangent direction at $p$ since three tangent directions must be shared among at least two components. Any point near $p$ in this distinguished tangent direction will be a point of positive type, which is a contradiction.
	
	Now we proceed by reverse induction on the number of points of $\Gamma$ of valence three, denoted $N_3(\Gamma)$. Observe that
	\benn
			N_3(\Gamma) = \sum_{\substack{p \in \Gamma \\ v(p) = 3}} 1 
			\leq \sum_{p \in \Gamma} (v(p) - 2) 
			= \deg K = 2g-2,
	\eenn
and equality is achieved precisely when $\Gamma$ is cubic. So if $N_3(\Gamma) = 2g-2$, then \eqref{Eq: Phi Inequality} holds by the main hypothesis of the lemma. 

	We now suppose that $N_3(\Gamma) < 2g-2$. There must exist a point $p_0 \in \Gamma$ of valence $\nu > 3$. For each $t > 0$, we define a new graph $\agam_t = (\Gamma_t, 0)$ as follows. Delete the point $p_0$ from $\Gamma$ and compactify the resulting graph by adjoining new limit points $x_1, \ldots, x_\nu$ at the ends of each of the edges that previously met at $p$. If $\Gamma \smallsetminus \{p_0\}$ is disconnected, we further assume that $x_1$ and $x_2$ belong to \textit{distinct} components of the compactification (after relabeling if necessary). The metric graph $\Gamma_t$ is obtained from this compactification by gluing $x_1$ and $x_2$ to one end of a closed segment of length $t$, and by gluing $x_3, \ldots, x_\nu$ to the other end of the segment. (See Figure~\ref{Fig: Cubic_Reduction}.) We call the segment $e_0 = \{\tilde{p}_1, \tilde{p}_2\}$, where $\tilde{p}_1$ is the image of $x_1$ and $x_2$, while $\tilde{p}_2$ is the image of $x_3, \ldots, x_\nu$. By construction $\Gamma_t$ is connected and has no points of positive type for each $t > 0$. There is a canonical deformation retraction $r_t: \Gamma_t \to \Gamma$ mapping the closed segment $e_0$ to $p_0$. 
	
	By construction the valence of $\tilde{p}_1$ is $3$. In particular, $N_3(\Gamma_t) > N_3(\Gamma)$, so the induction hypothesis implies
	\benn
		\Zphi{\agam_t} \geq c(g) \ell(\Gamma_t) = c(g) \ell(\Gamma) + c(g)t.
	\eenn
Therefore, to complete the proof it suffices to show that 
	\benn
		\Zphi{\agam} = \lim_{t \to 0} \Zphi{\agam_t}.
	\eenn
But this is precisely the content of Proposition~\ref{Prop: Continuity}. We have now completed the induction step. 
\end{proof} 

\begin{figure}[htb]
	\begin{picture}(100,48)(107,-5)

		\put(0,0){\includegraphics{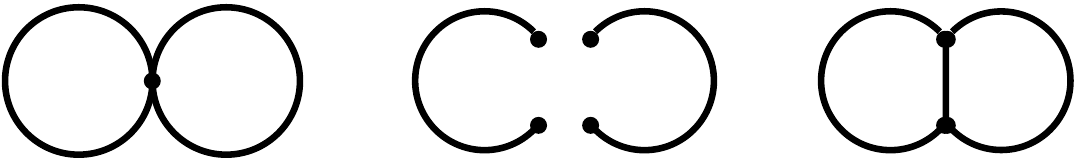}}
		\put(85,-5){$\agam$}
		\put(50,20){$p_0$}
		\put(98, 20){$\leadsto$}
		\put(141, 30){$x_1$}
		\put(176, 30){$x_2$}
		\put(141, 14){$x_3$}
		\put(174, 14){$x_4$}
		\put(216, 20){$\leadsto$}
		\put(308,-5){$\agam_t$}
		\put(259, 29){$\tilde{p}_1$}
		\put(281, 12){$\tilde{p}_2$}
		
	\end{picture}
	
	\caption{Here is an illustration of the construction of $\agam_t$ given in the proof of 
		Lemma~\ref{Lem: Reduction Lemma} in the case where the valence of $p_0$ is $\nu = 4$.} 
	\label{Fig: Cubic_Reduction}
\end{figure}
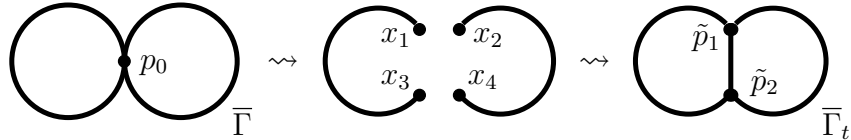

\begin{proof} [Proof of Proposition~\ref{Reduce conjecture}]
	Apply the sequence of Lemmas~\ref{Lem: All type 0}, \ref{Lem: q equals 0}, and~\ref{Lem: Reduction Lemma}.
\end{proof}


\subsection{Strategy for the Proof of Theorem~\ref{Thm: Graph Bounds}}
\label{Sec: Strategy}

	We have seen in Proposition~\ref{Reduce conjecture} that, to prove Theorem~\ref{Thm: Graph Bounds}, it suffices to prove it when $\agam = (\Gamma, 0)$ is an irreducible cubic pm-graph of genus $g \geq 2$. Observe that such a graph admits a \textbf{canonical model} $G^{\circ}$ with precisely $2g-2$ vertices. Indeed, we may take the set of valence $3$ points as a vertex set, and the calculation in the proof of Lemma~\ref{Lem: Reduction Lemma} shows this set has cardinality $2g-2$. This model may have multiple edges, but it cannot have loop edges by irreducibility. Note further that $g = \#E(G^\circ) - \#V(G^\circ) + 1$, which shows
	\benn
		\#V(G^\circ) = 2g-2 \qquad \#E(G^\circ) = 3g - 3.
	\eenn
	
	The next step is to list all possible combinatorial graphs that can arise as the canonical model $G^{\circ}$ as above, up to combinatorial isomorphism. One simple way to do it is to note that such a graph admits a maximal spanning tree consisting of $2g-2$ vertices and $2g-3$ edges. It is easy to write down all of the trees of this type for small values of $g$. Given any such tree we build candidates for $G^\circ$ by adjoining $g$ edges to it. Many of these will yield reducible graphs or graphs that fail to be $3$-regular. Of the remaining candidates, many will be isomorphic. Rather than go through the details, we simply give the complete list of such graphs up to combinatorial isomorphism in Figure~\ref{Fig: All Graphs} below.
	
	Now let us suppose we have a combinatorial graph $G^{\circ}$ constructed as in the last paragraph. It is the canonical model of some metric graph $\agam$ with edge lengths $\ell_1, \ldots, \ell_m$, where $m = 3g-3$. By Proposition~\ref{Prop: Special Forms}, we know that
	\be \label{Eqn: Special Forms}
		\ba
			\Zphi{\agam} &= \frac{g-1}{6g} \ell(\Gamma) 
				- \frac{\omega_1(\ell_1, \ldots, \ell_m)}{\eta(\ell_1, \ldots, \ell_m)^2}.
		\ea
	\ee
	
	Finally, we are reduced to the problem of bounding the quotient $\omega_1 / \eta^2$. The next proposition illustrates the strategy we will take for obtaining such a bound.
		
\begin{prop} \label{Prop: Verify}
	Let $\agam = (\Gamma, 0)$ be an irreducible cubic pm-graph of genus $g \geq 2$. 
	Let $G^{\circ}$ be its canonical model with edges $e_1, \ldots, e_m$. Let $\sigma_2$ and $\sigma_3$ be 
	the second and third symmetric polynomials in the lengths $\ell_1, \ldots, \ell_m$.
	If there exists a positive constant $A$ such that
		\be \label{Eqn: A Constant}
			\frac{\omega_1(\ell_1, \ldots, \ell_m)}{\eta(\ell_1, \ldots, \ell_m)^2} 
				\leq A \frac{\sigma_3(\ell_1, \ldots, \ell_m)}{\sigma_2(\ell_1, \ldots, \ell_m)},
		\ee
	then 
		\benn
			\Zphi{\agam}   \geq  \left[ \frac{g-1}{6g} 	- A\frac{3g-5}{9(g-1)} \right]\ell(\Gamma).				\eenn
\end{prop}

	Before beginning the proof, let us indicate how one might use the proposition to find the desired constant $A$. Clearing denominators in \eqref{Eqn: A Constant}, we find the inequality
	\benn
		A  \sigma_3(\ell_1, \ldots, \ell_m)  \eta(\ell_1, \ldots, \ell_m)^2 
			- \omega_1(\ell_1, \ldots, \ell_m) \sigma_2(\ell_1, \ldots, \ell_m) \geq 0.
	\eenn 
The left side is a polynomial $\sum_{\alpha} f_\alpha(A) \ell_\alpha$, where $\alpha$ is a multi-index and $f_\alpha$ is a linear polynomial in $A$ with rational coefficients. If there exists a value of $A$ such that all of these linear polynomials are nonnegative, then of course $\sum_{\alpha} f_\alpha(A) \ell_\alpha \geq 0$ and the desired inequality holds. In order to find $A$, we solve the system of linear inequalities $f_\alpha(A) \geq 0$. As there can be thousands of such inequalities, we utilize \textit{Mathematica} for this step. 
\begin{withcode} The computations are summarized in the next section, and the \textit{Mathematica} notebooks are recreated afterward. \end{withcode}
\begin{withoutcode} The computations are summarized in the next section. \end{withoutcode}

\begin{proof}[Proof of Proposition~\ref{Prop: Verify}]
	We begin by showing the inequality
	\be \label{Eqn: Symmetric}
		\frac{\sigma_{3}}{\sigma_2}\leq \frac{m-2}{3m}\ell(\Gamma).
	\ee
We use the following well-known generalization of the arithmetic-geometric mean inequality:
		\benn
			\left[\frac{\sigma_1}{\binom{m}{1}}\right]^{1} \geq 
			\left[\frac{\sigma_2}{\binom{m}{2}}\right]^{1/2} \geq
			\left[\frac{\sigma_3}{\binom{m}{3}}\right]^{1/3} \geq \cdots \geq
			\left[\frac{\sigma_m}{\binom{m}{m}}\right]^{1/m},
		\eenn
where $\sigma_i$ is the $i$th symmetric polynomial in $\ell_1, \ldots, \ell_m$. Using the first two inequalities, we find
	\benn
		\frac{\sigma_1}{\binom{m}{1}}\frac{\sigma_2}{\binom{m}{2}}\geq
		\left[\frac{\sigma_3}{\binom{m}{3}}\right]^{1/3}\left[\frac{\sigma_3}{\binom{m}{3}}\right]^{2/3}
		= \frac{\sigma_3}{\binom{m}{3}}.
	\eenn
Rearranging and expanding out the binomial coefficients gives~\eqref{Eqn: Symmetric}, using $\sigma_1 = \ell(\Gamma)$.

	Now recall that the number of edges of $G^\circ$ is $m = 3g - 3$. Combining this fact with~\eqref{Eqn: A Constant} and~\eqref{Eqn: Symmetric} shows
	\benn
		\frac{\omega_1(\ell_1, \ldots, \ell_m)}{\eta(\ell_1, \ldots, \ell_m)^2} 
						\leq A \frac{\sigma_3(\ell_1, \ldots, \ell_m)}{\sigma_2(\ell_1, \ldots, \ell_m)} 
					\leq A\frac{3g-5}{9(g-1)}\ell(\Gamma).
	\eenn
The proof is complete upon applying this last inequality to \eqref{Eqn: Special Forms}.
\end{proof}

	We have now given the complete strategy for proving Theorem~\ref{Thm: Graph Bounds}, and all that remains is to provide the computations. To conclude this part, we give a general conjecture on the constant $A$ in the proposition, and consequently a strong conjectural lower bound for $\Zphi{\agam}$. 
		
\begin{conjecture} \label{Conj: Big Conjecture}
	With the notation of Proposition~\ref{Prop: Verify}, one may take
		\benn
			A = \frac{7(g-1)^2}{6g(3g-5)}.
		\eenn
	That is, Conjecture~\ref{Conj: Zhang} holds with $c(g) = (g-1)/ 27g$.
\end{conjecture}

	This conjecture is born entirely from empirical data. It holds for all pm-graphs of genus $g = 2, 3, 4$, although it does not yield the smallest possible constant $A$ in some cases. See Table~\ref{Table: Data}.

\subsection{Computational Data}
\label{Sec: Data}

	Here we provide all of the necessary data to complete the proof of Thereom~\ref{Thm: Graph Bounds}. First, we enumerate the possible isomorphism classes of irreducible cubic combinatorial graphs $G$ that can arise as the canonical model of a pm-graph $\agam = (\Gamma, 0)$ of genus $ 2 \leq g \leq 4$. As we have already indicated the strategy for finding all such graphs in the previous section, we simply list them in Figure~\ref{Fig: All Graphs}.

\begin{figure}[p] 

	\begin{picture}(200,420)(100,-15)

		\put(25,-10){\includegraphics{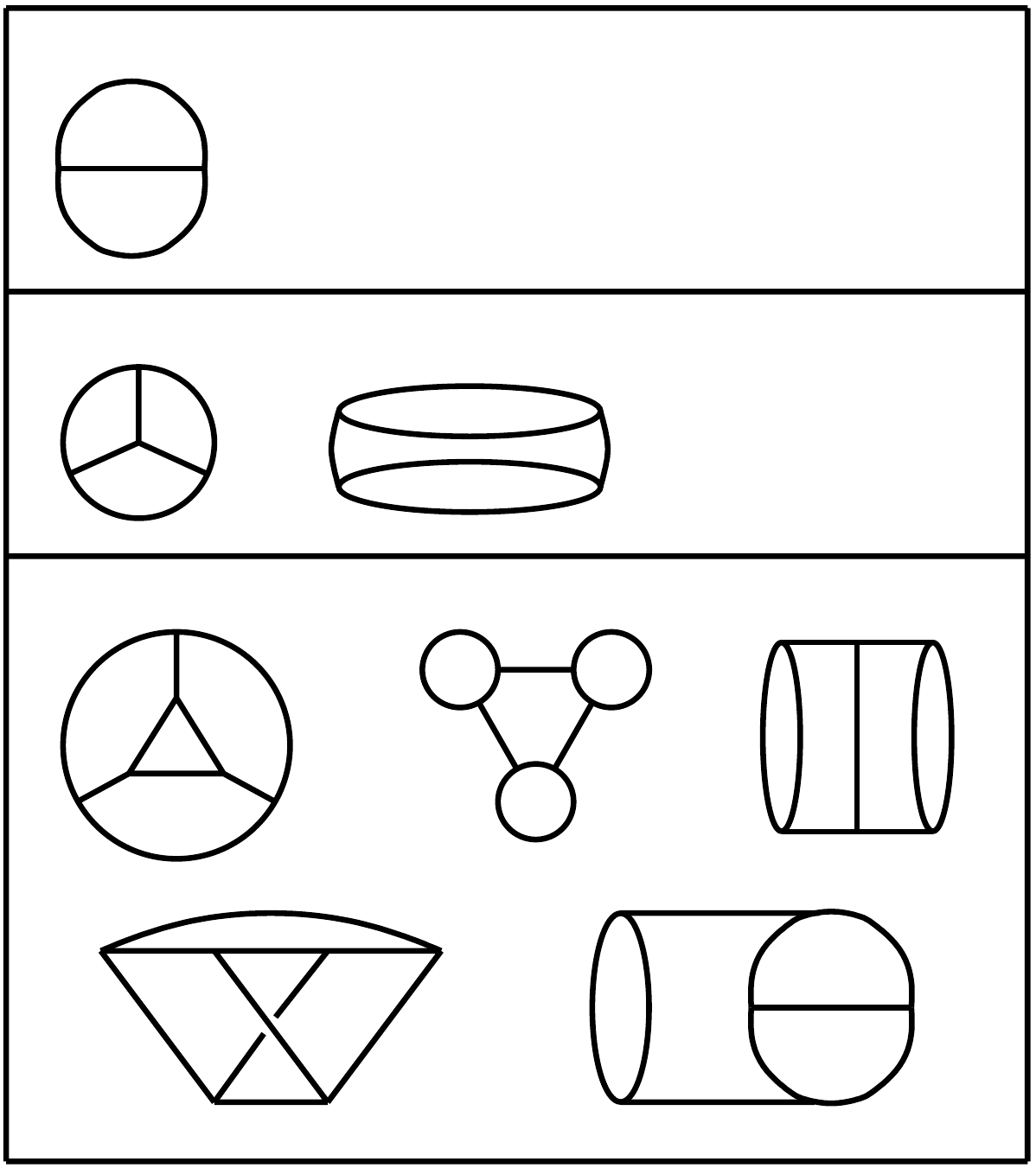}}
		
		\put(40, 366){\textbf{Genus $2$}}
		\put(40, 271){\textbf{Genus $3$}}
		\put(40, 182){\textbf{Genus $4$}}
		
		\put(95, 295){\textbf{$\agam_2$}}
		\put(96, 205){\textbf{$\agam_{3,1}$}}
		\put(230, 205){\textbf{$\agam_{3,2}$}}	
		\put(119, 98){\textbf{$\agam_{4,1}$}}
		\put(221, 98){\textbf{$\agam_{4,2}$}}
		\put(345, 98){\textbf{$\agam_{4,3}$}}
		\put(145, 8){\textbf{$\agam_{4,4}$}}
		\put(326, 8){\textbf{$\agam_{4,5}$}}
			
	\end{picture}
	
	\caption{Here we have a complete list of irreducible cubic pm-graphs with $q \equiv 0$ and genus at most $4$, 
		up to segment scaling and topological isomorphism. Equivalently, one can view this list as giving all 
		possible combinatorial graphs (up to isomorphism) that can arise as the canonical model of
		 such a pm-graph.}
	\label{Fig: All Graphs}
	
\end{figure}

	Table~\ref{Table: Data} summarizes the data obtained for the graphs in Figure~\ref{Fig: All Graphs}. All of the computation was performed in \textit{Mathematica} using the algorithms suggested in Section~\ref{Sec: Rapid}. The calculations are exact in the sense that no numerical methods were used. 
\begin{withoutcode}
To reiterate, the code used for these calculations can be found appended to the end of the arXiv edition of this article \cite{Faber_EGBC_arXiv_2008}.
\end{withoutcode}
	
	As there is only one graph of genus $2$ in the table, we see immediately that Theorem~\ref{Thm: Graph Bounds} holds with $c(2) = 1/27$. Both graphs of genus $3$ yield the same lower bound for $\Zphi{\agam}$, so we may take $c(3) = 2/81$ in the theorem. For genus $4$, two different lower bounds for $\Zphi{\agam}$ were obtained depending on whether the reductions in the previous section yield the graph $\agam_{4,2}$. Of course we take the smaller of the two and conclude that $c(4) = 1/36$ will suffice in the theorem. 
	
\begin{withcode}
	Following Table~\ref{Table: Data} is the \textit{Mathematica} code used for these computations. Some of the functions that appear there are useful for experimenting, but were not necessarily used to produce the results in this article. Finally, we reproduce the \textit{Mathematica} notebooks in which the computations were executed.
\end{withcode}
	
\begin{table}[p]

	\begin{tabular}{|c|c|c|}
		\hline
		$\stackrel{\phantom{=}}{\agam}$ & $A$ & $\frac{g-1}{6g} - A \frac{3g-5}{9(g-1)}$ \\
		\hline
		\hline
		$\stackrel{\phantom{=}}{\agam_2}$ & $5/12$ & $1 / 27$ \\
		\hline
		\hline
		$\stackrel{\phantom{=}}{\agam_{3,1}}$ & $7 / 18$ & $ 2/ 81$  \\
		\hline
		$\stackrel{\phantom{=}}{\agam_{3,2}}$ & $7 / 18$ & $2 / 81$  \\
		\hline
		\hline
		$\stackrel{\phantom{=}}{\agam_{4,1}}$& $3/8$ & $1/36 \approx 0.0278$  \\
		\hline
		$\stackrel{\phantom{=}}{\agam_{4,2}}$& $7/26$ & $155/2808 \approx 0.0552$  \\
		\hline
		$\stackrel{\phantom{=}}{\agam_{4,3}}$ & $3/8$ & $1/36$  \\
		\hline
		$\stackrel{\phantom{=}}{\agam_{4,4}}$ & $3/8$ & $1/36$  \\
		\hline
		$\stackrel{\phantom{=}}{\agam_{4,5}}$& $3/8$ & $1/36$  \\
		\hline
	\end{tabular}

	\vspace{0.5cm}
	
	\caption{The first column lists the names of the irreducible cubic pm-graphs given in 
		Figure~\ref{Fig: All Graphs}. The second column gives the smallest value of $A$ that
		satisfies the inequality \eqref{Eqn: A Constant} of Proposition~\ref{Prop: Verify}, while the
		third column gives the value of the constant provided by the conclusion of the proposition.}
	\label{Table: Data}
\end{table}

\begin{withcode}

\begin{figure*}[p] 

	\begin{picture}(200,420)(0,0)

		\put(-195,-290){\includegraphics{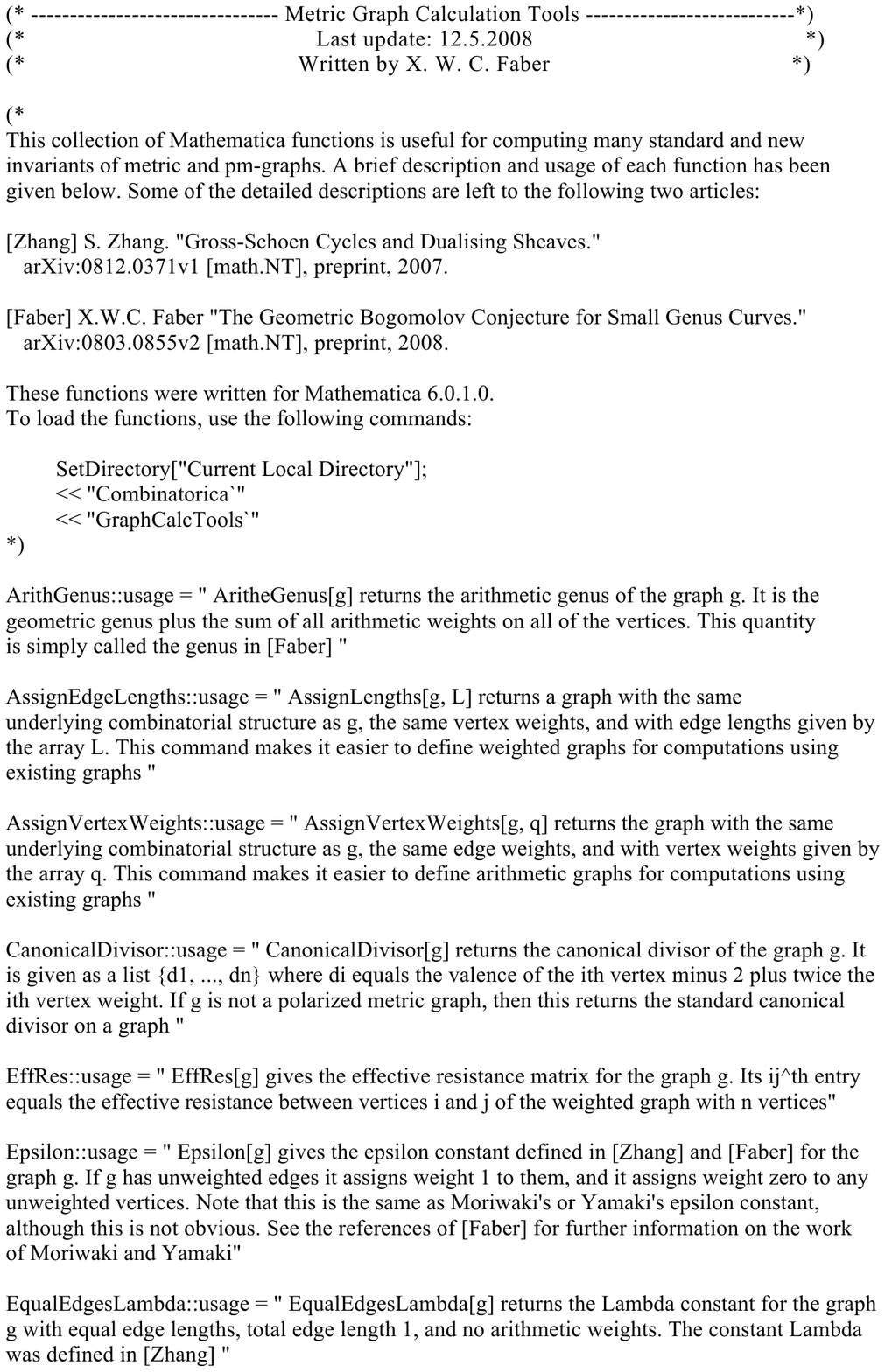}}

	\end{picture}
	
\end{figure*}

\begin{figure*}[p] 

	\begin{picture}(200,420)(0,0)

		\put(-195,-290){\includegraphics{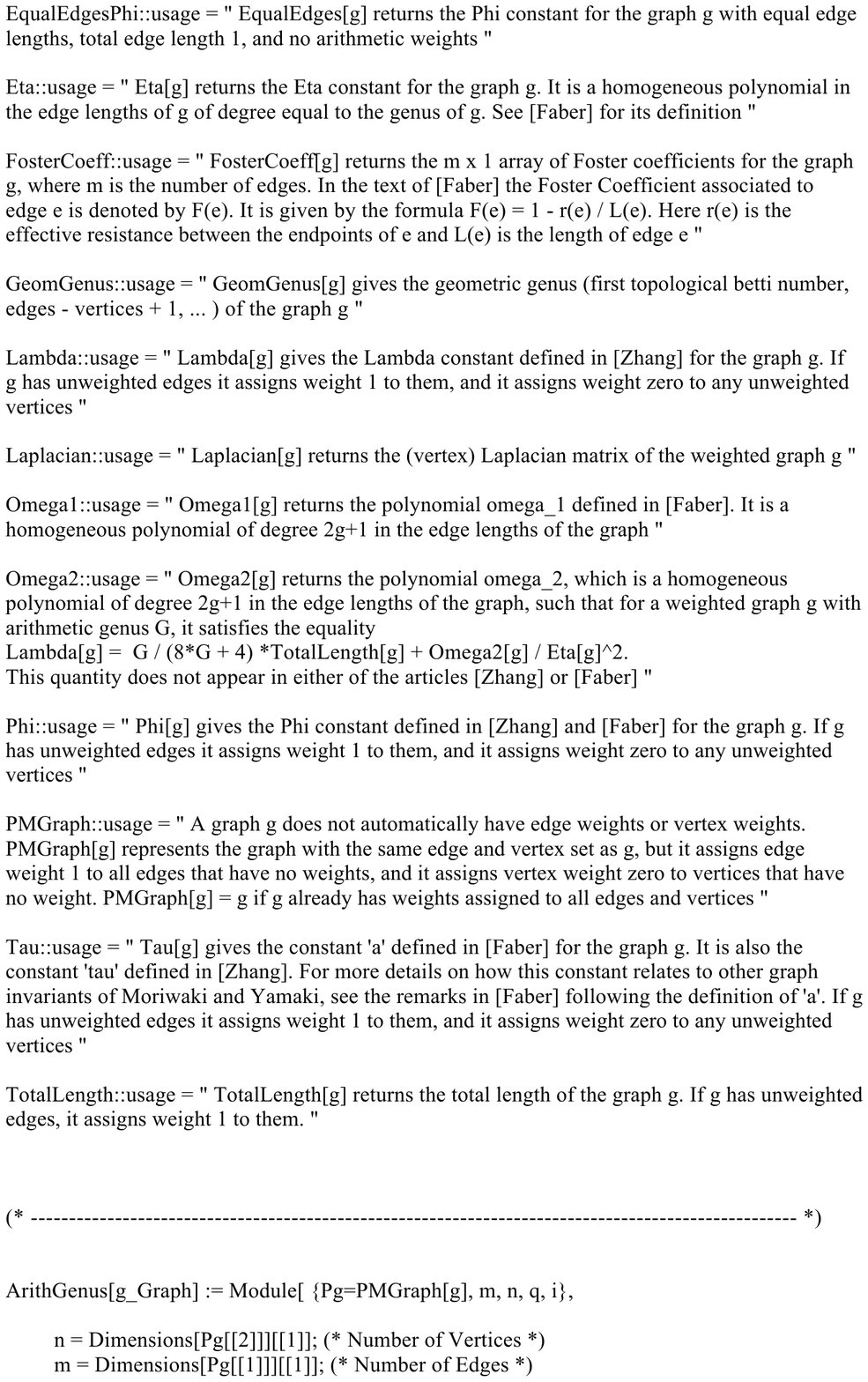}}

	\end{picture}
	
\end{figure*}

\begin{figure*}[p] 

	\begin{picture}(200,420)(0,0)

		\put(-195,-290){\includegraphics{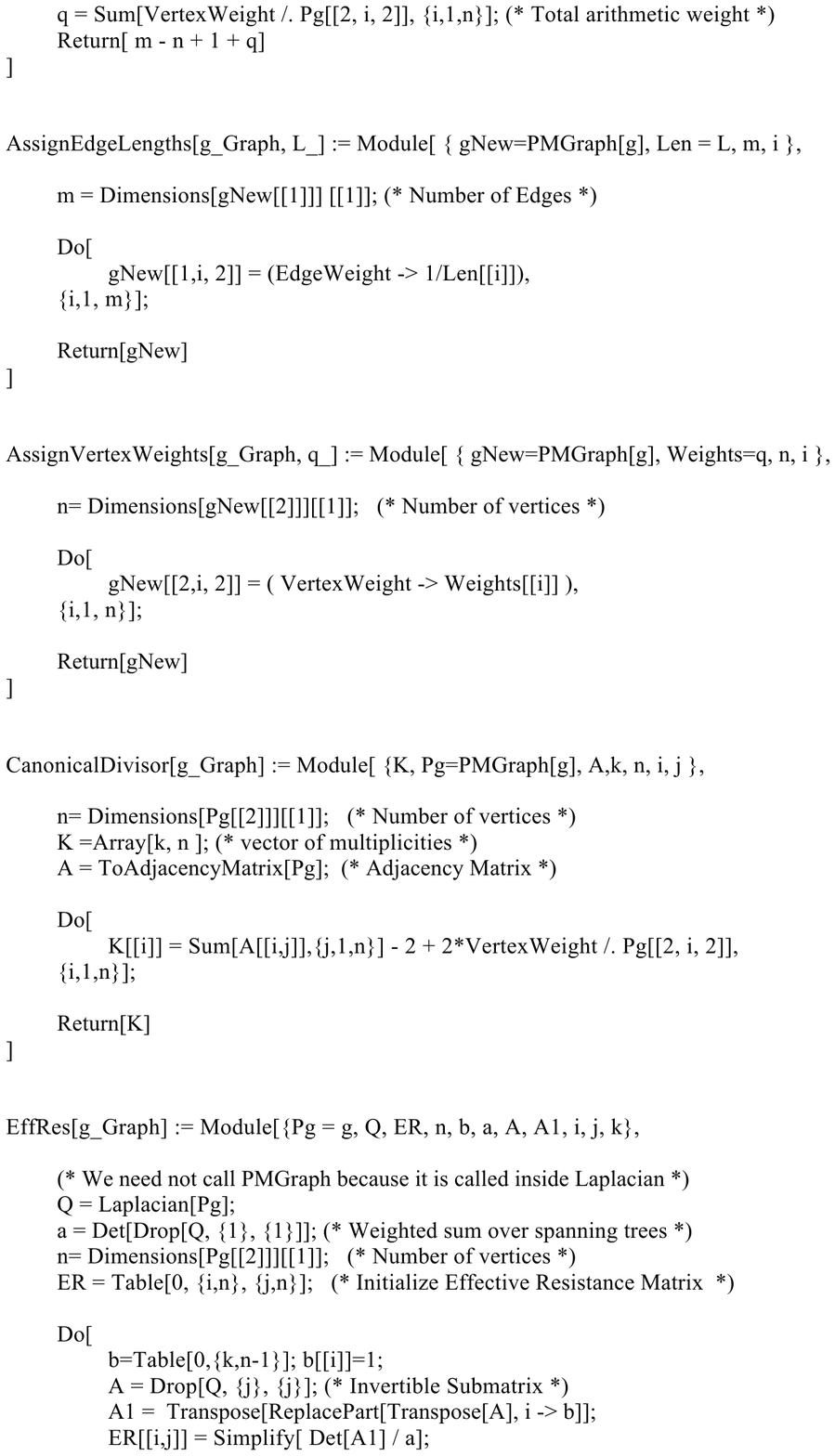}}

	\end{picture}
	
\end{figure*}

\begin{figure*}[p] 

	\begin{picture}(200,420)(0,0)

		\put(-195,-290){\includegraphics{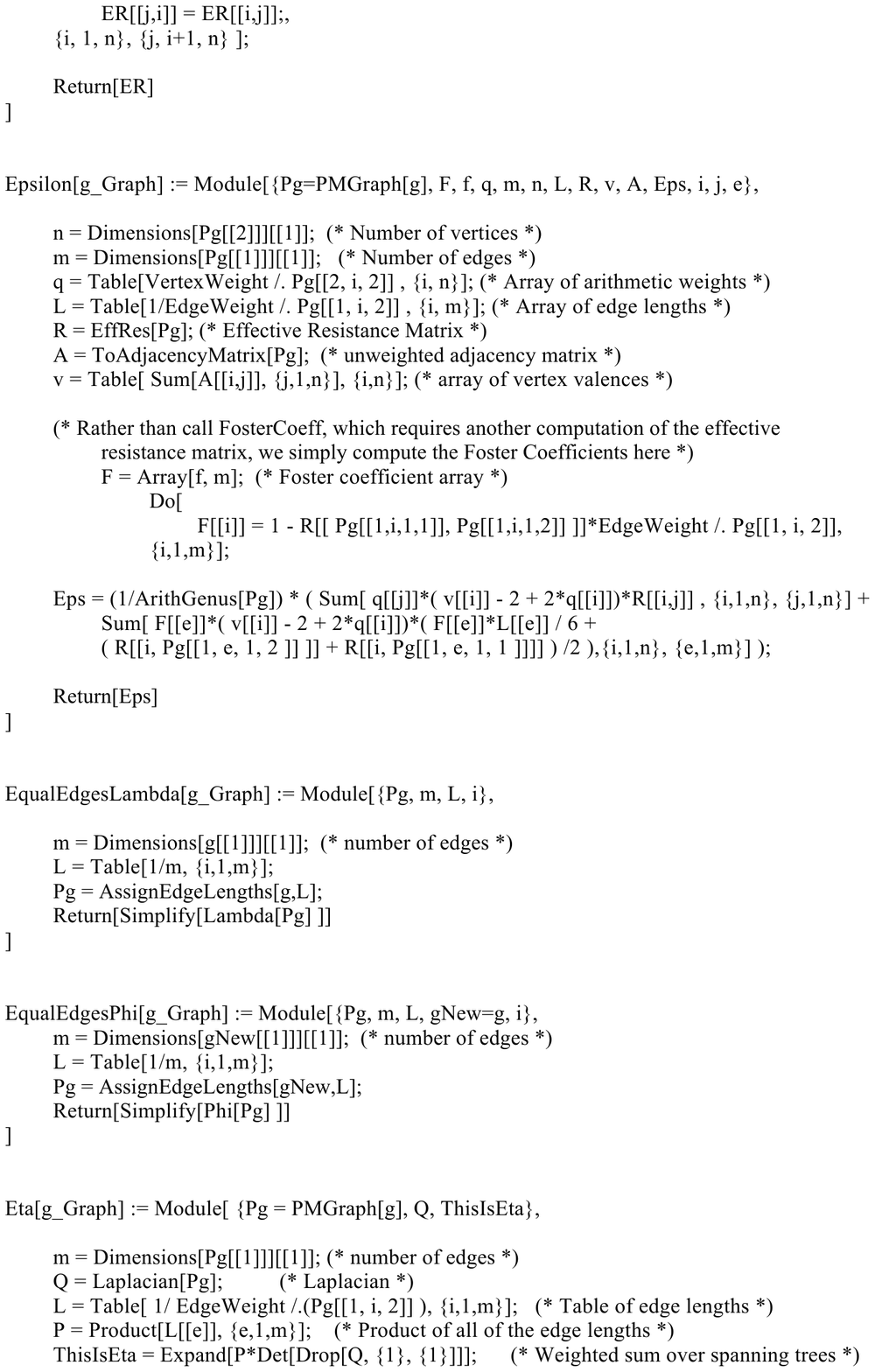}}

	\end{picture}
	
\end{figure*}

\begin{figure*}[p] 

	\begin{picture}(200,420)(0,0)

		\put(-195,-290){\includegraphics{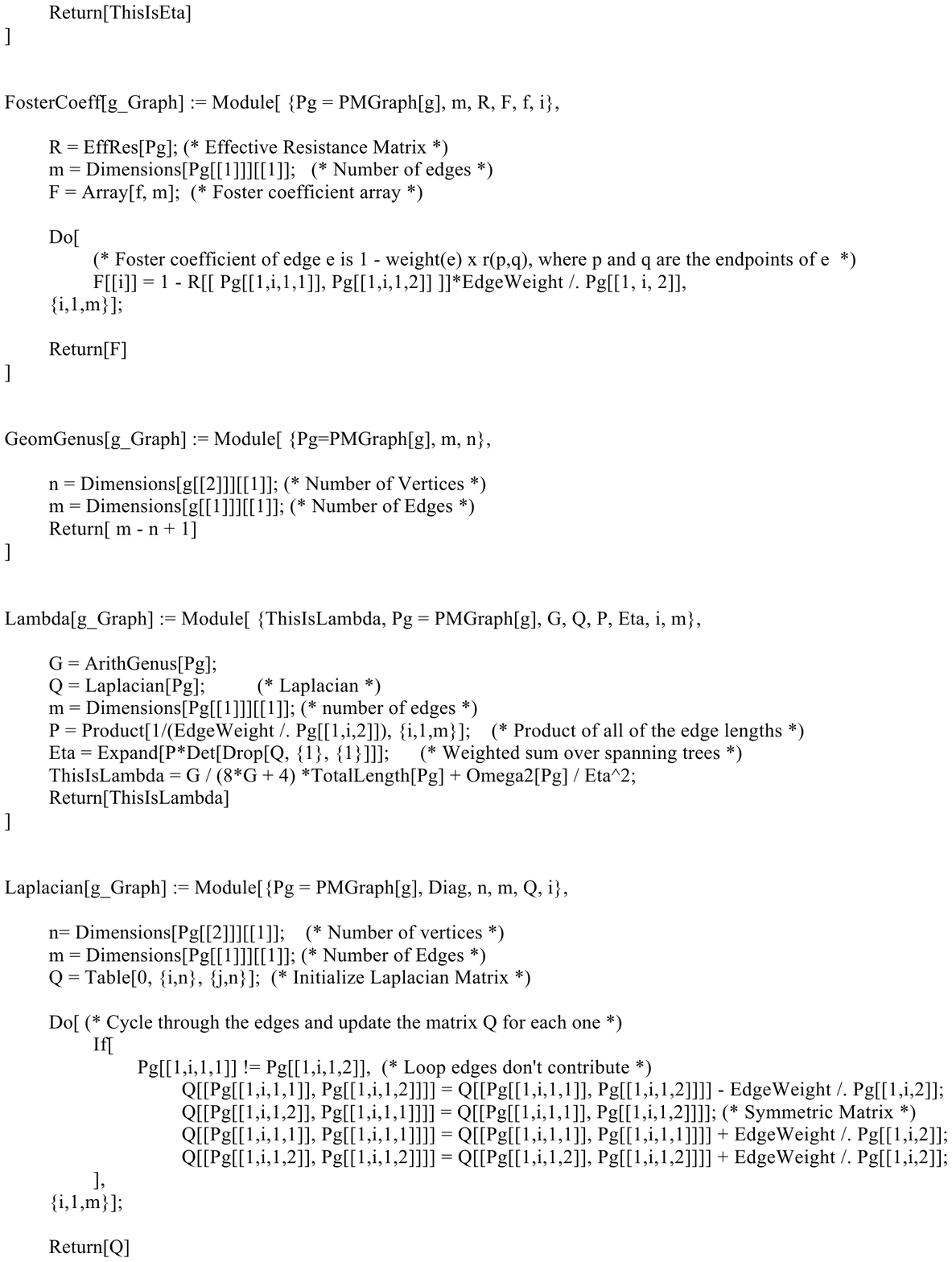}}

	\end{picture}
	
\end{figure*}

\begin{figure*}[p] 

	\begin{picture}(200,420)(0,0)

		\put(-195,-290){\includegraphics{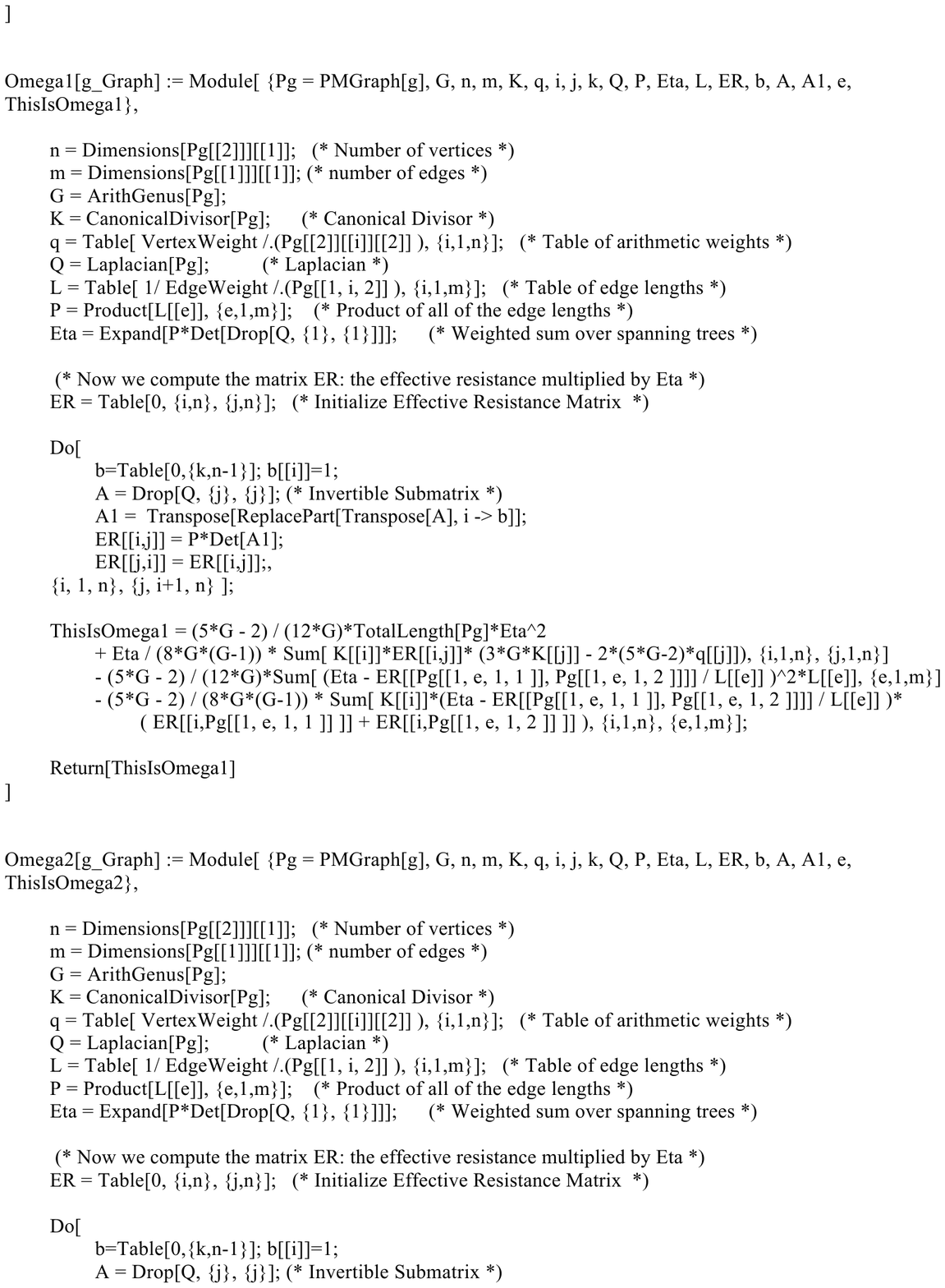}}

	\end{picture}
	
\end{figure*}

\begin{figure*}[p] 

	\begin{picture}(200,420)(0,0)

		\put(-195,-290){\includegraphics{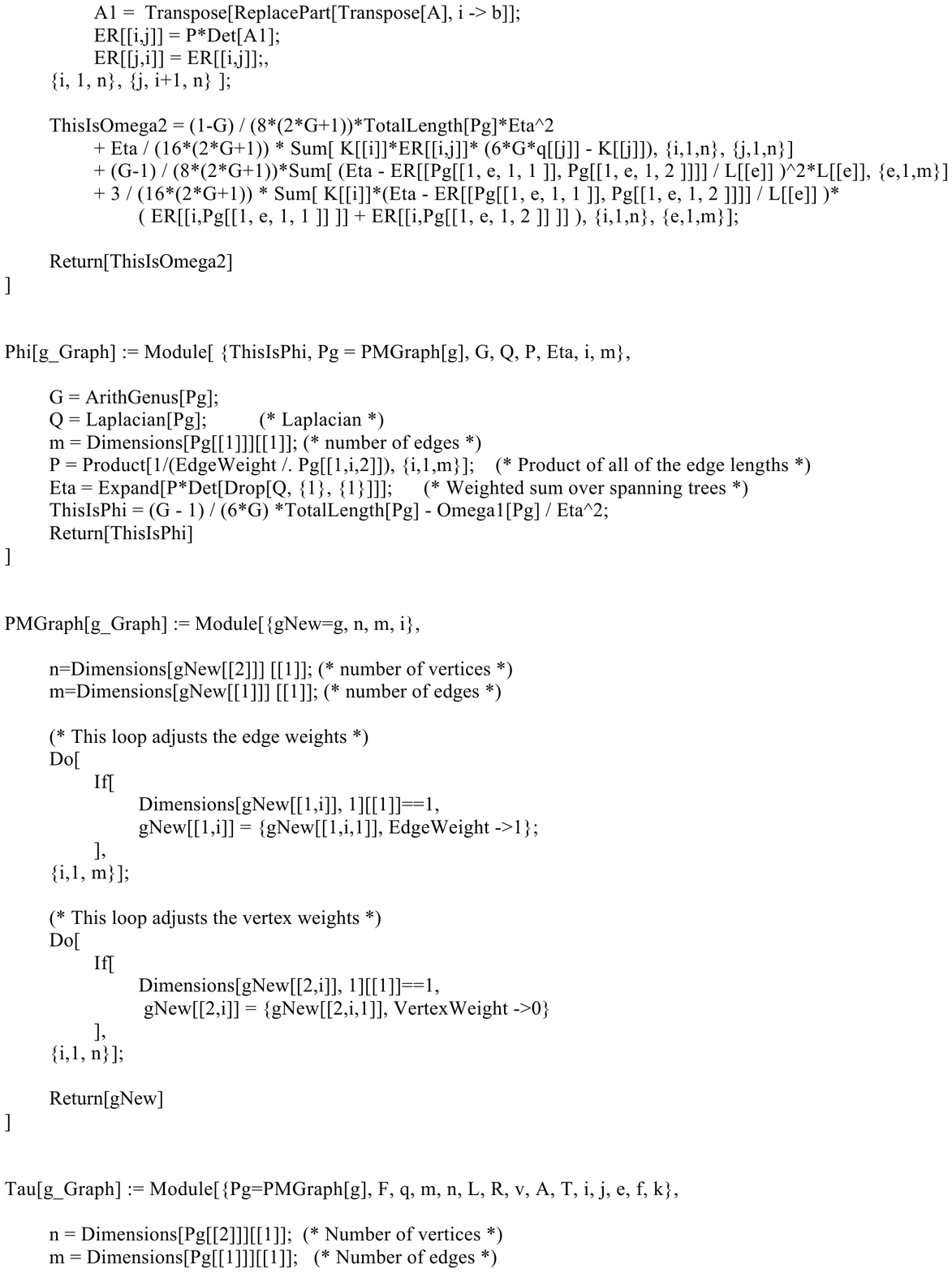}}

	\end{picture}
	
\end{figure*}

\begin{figure*}[p] 

	\begin{picture}(200,420)(0,0)

		\put(-195,-290){\includegraphics{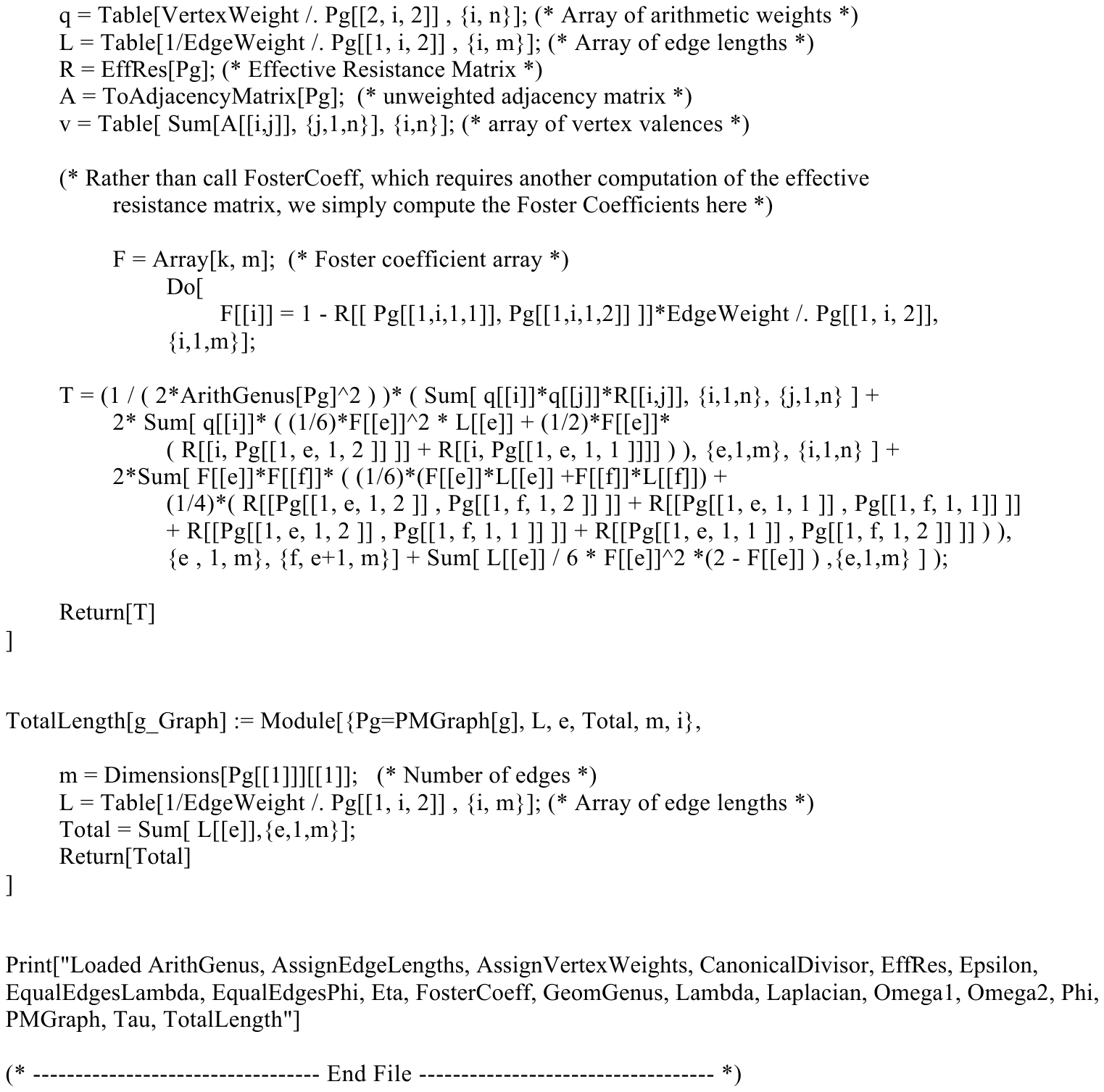}}

	\end{picture}
	
\end{figure*}

\begin{figure*}[p] 

	\begin{picture}(200,420)(0,0)

		\put(-200,-290){\includegraphics{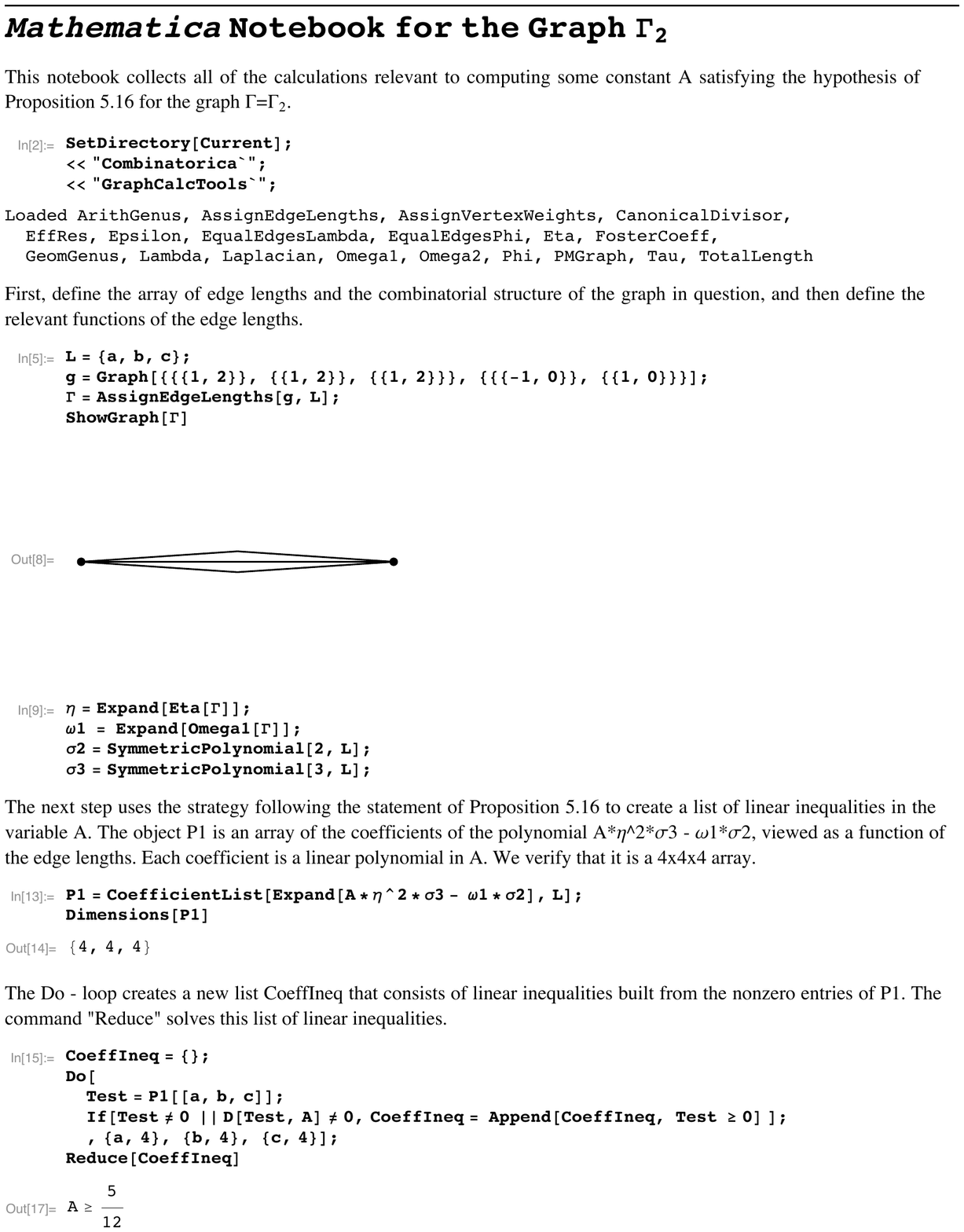}}

	\end{picture}
	
\end{figure*}

\begin{figure*}[p] 

	\begin{picture}(200,420)(0,0)

		\put(-200,-290){\includegraphics{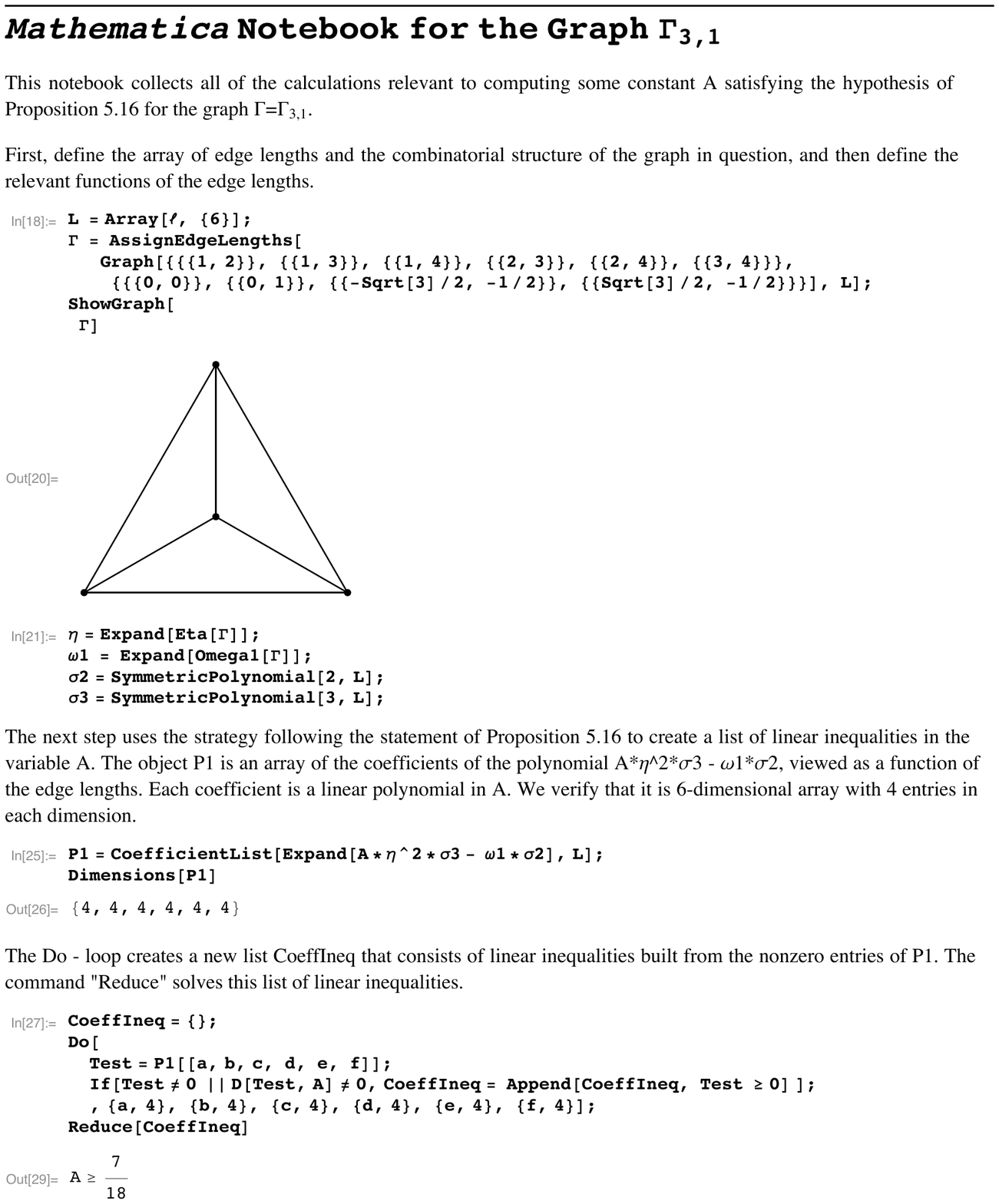}}

	\end{picture}
	
\end{figure*}

\begin{figure*}[p] 

	\begin{picture}(200,420)(0,0)

		\put(-200,-290){\includegraphics{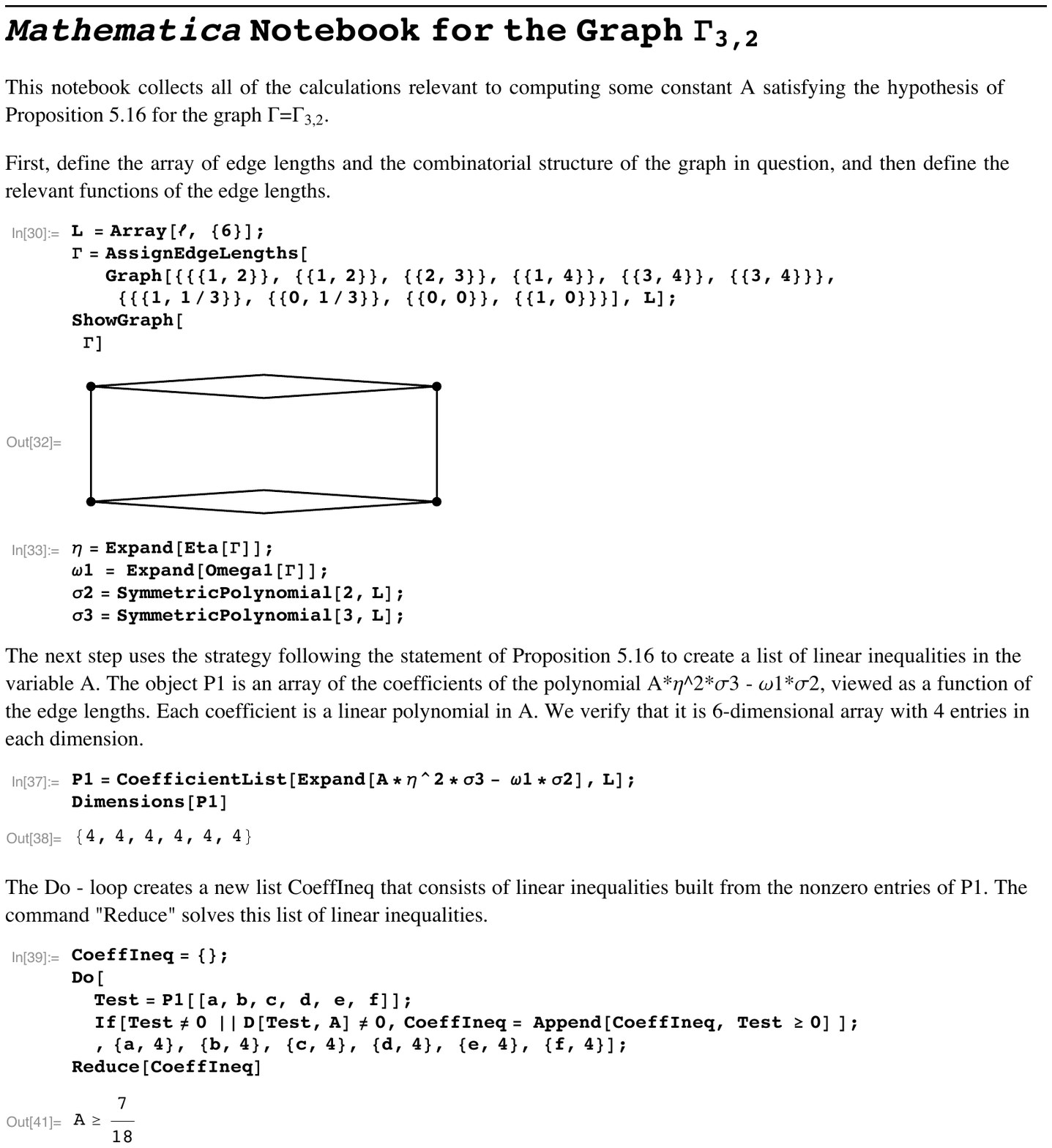}}

	\end{picture}
	
\end{figure*}

\begin{figure*}[p] 

	\begin{picture}(200,420)(0,0)

		\put(-200,-290){\includegraphics{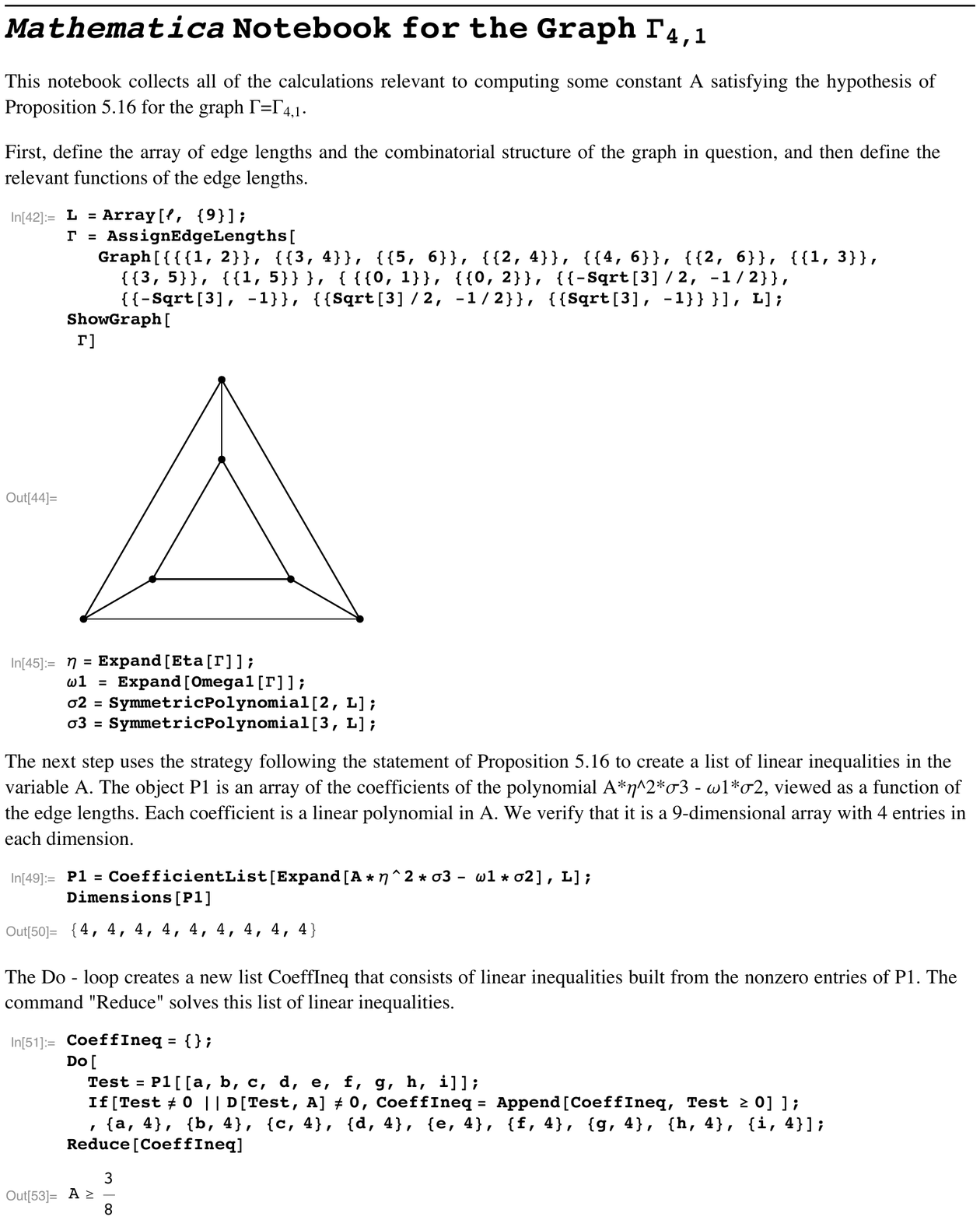}}

	\end{picture}
	
\end{figure*}

\begin{figure*}[p] 

	\begin{picture}(200,420)(0,0)

		\put(-200,-290){\includegraphics{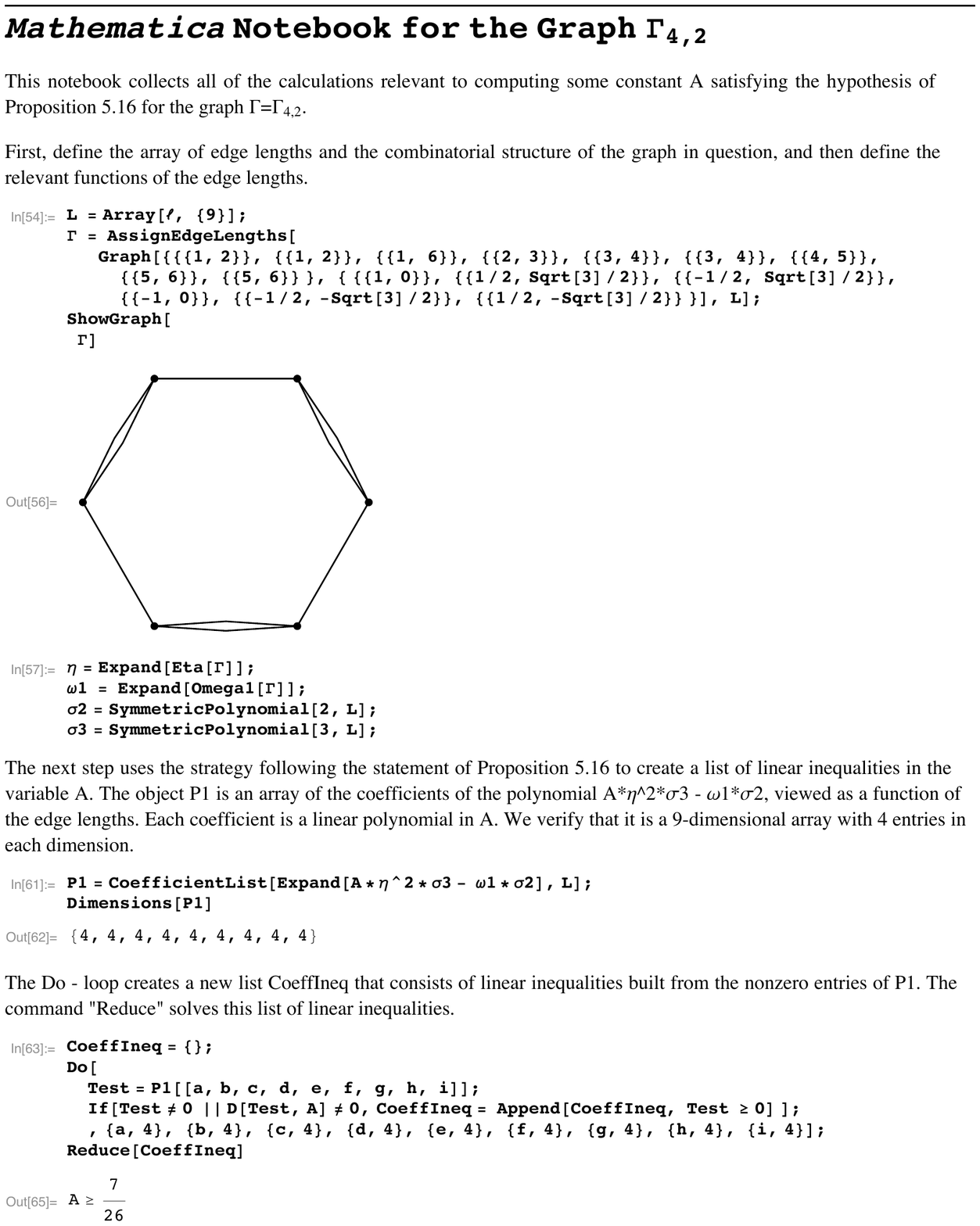}}

	\end{picture}
	
\end{figure*}

\begin{figure*}[p] 

	\begin{picture}(200,420)(0,0)

		\put(-200,-290){\includegraphics{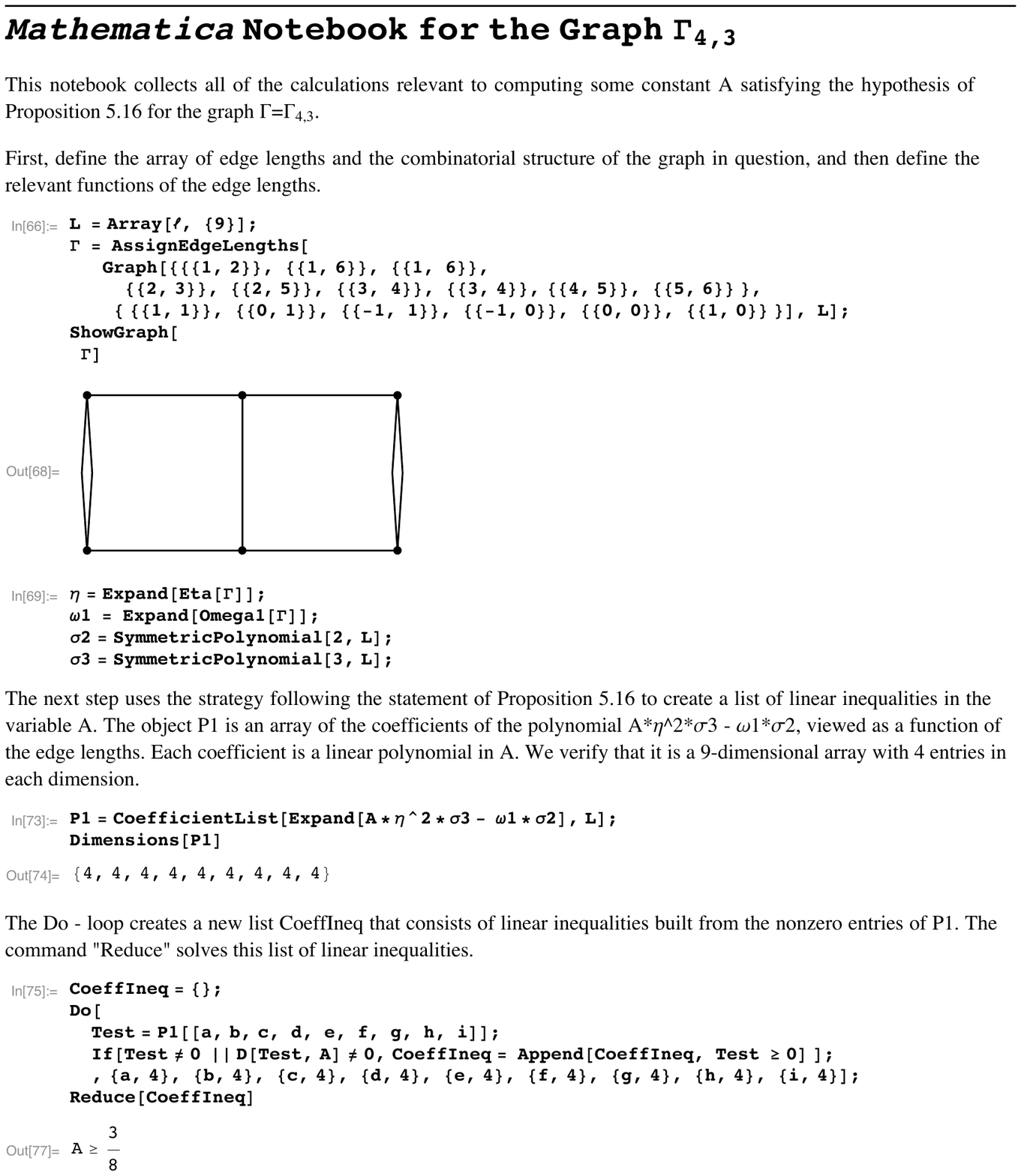}}

	\end{picture}
	
\end{figure*}

\begin{figure*}[p] 

	\begin{picture}(200,420)(0,0)

		\put(-200,-290){\includegraphics{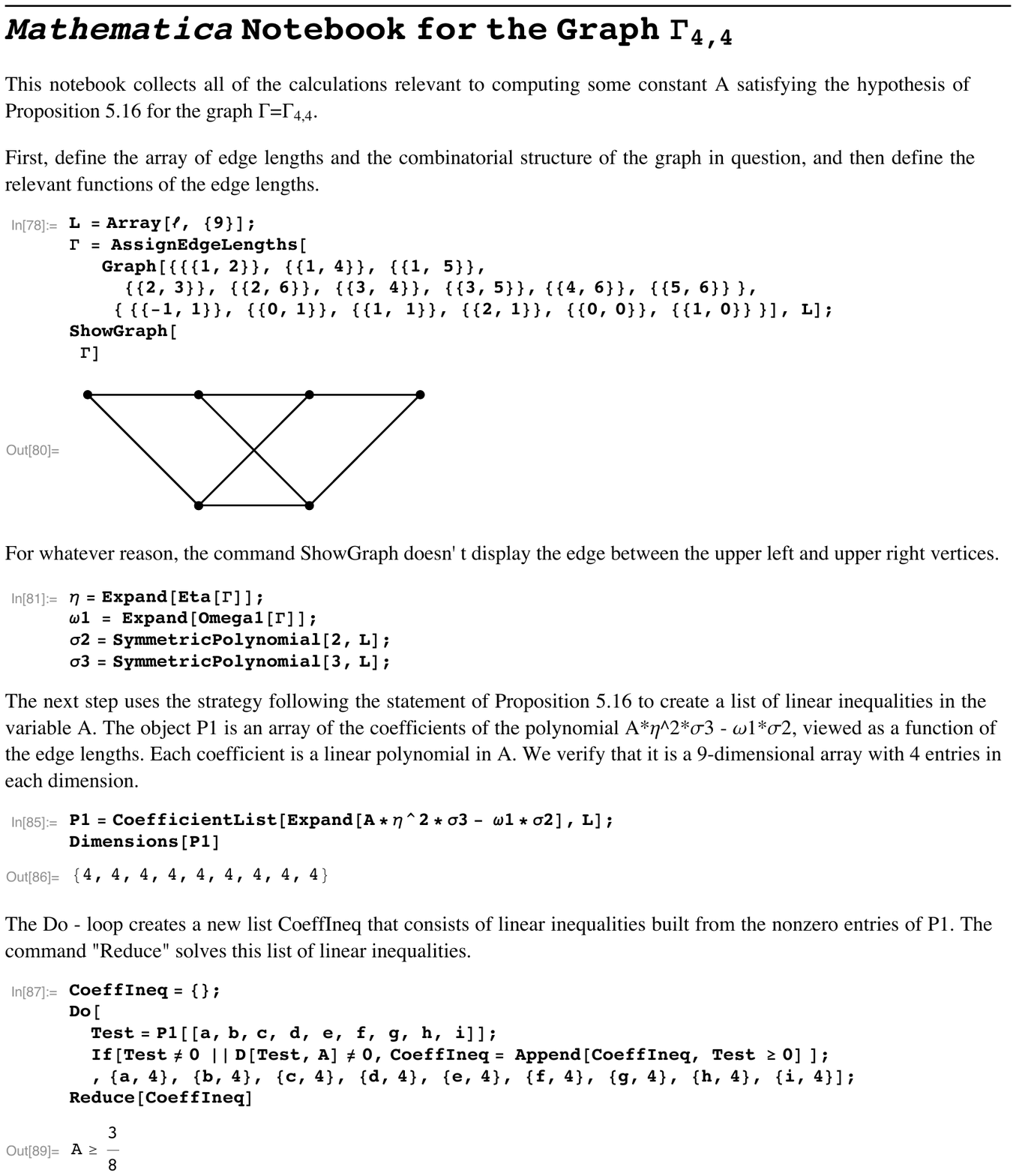}}

	\end{picture}
	
\end{figure*}

\begin{figure*}[p] 

	\begin{picture}(200,420)(0,0)

		\put(-200,-290){\includegraphics{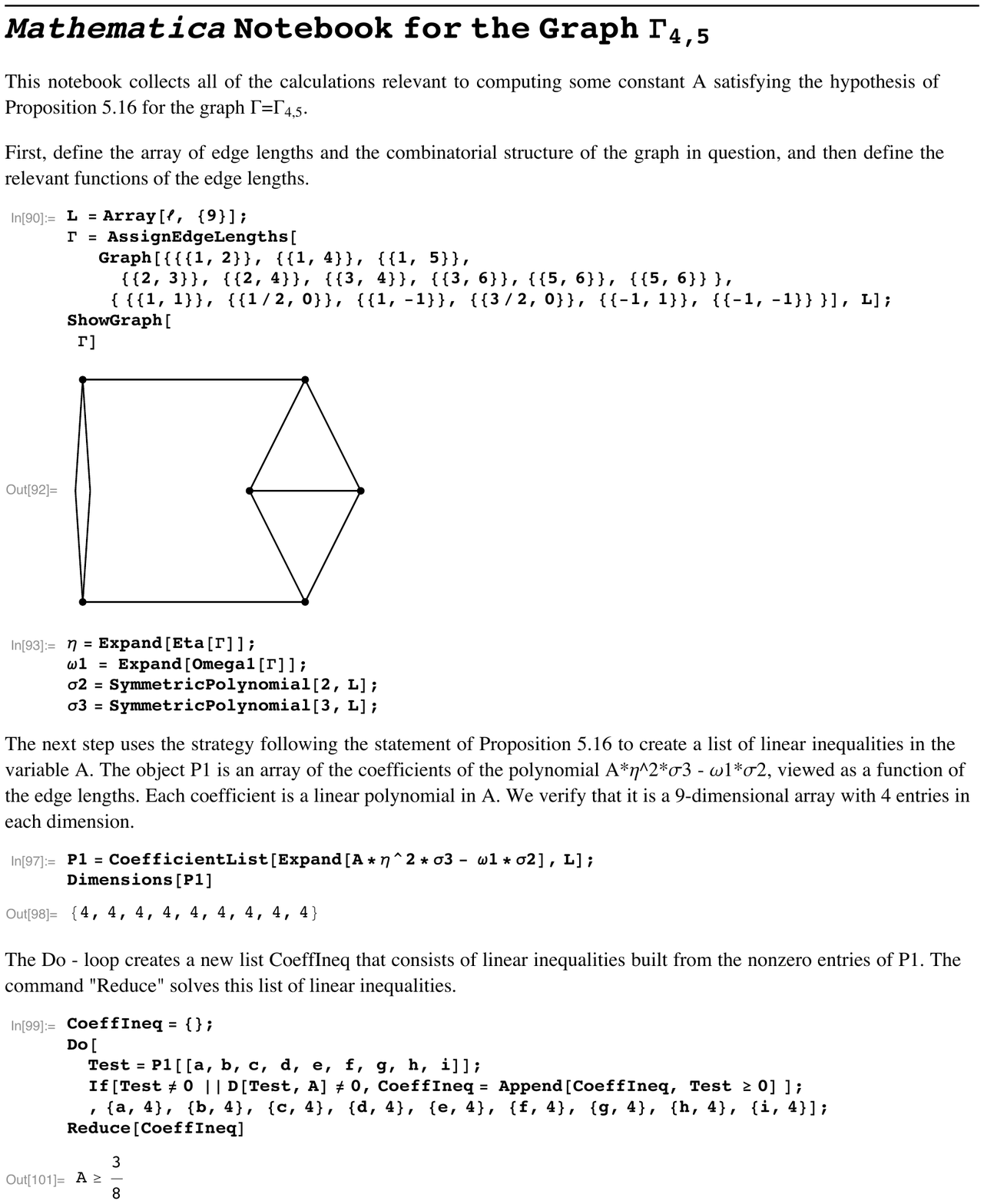}}

	\end{picture}
	
\end{figure*}

\end{withcode}


\bibliographystyle{abbrv}
\bibliography{xander_bib}

\end{document}